\documentclass[11pt, a4paper]{amsart}

\usepackage[T1]{fontenc}     
\usepackage{amsfonts}        
\usepackage{listings}
\usepackage{color} 
\usepackage{fancyvrb}
\usepackage{amsmath}
\usepackage{mathrsfs}
\usepackage{amsthm}          
\usepackage{graphicx}		
\usepackage{enumerate}   
\usepackage{bbm}
\usepackage{euscript}
\usepackage{yhmath}
\usepackage{MnSymbol}
\usepackage[colorlinks=true,
            linkcolor=blue,  
            citecolor=blue,  
            urlcolor=blue]{hyperref} 
\usepackage[a4paper, margin=1in]{geometry}
  
\newcommand{\R}{\mathbb{R}}
\newcommand{\Cp}{\mathbb{C}}
\newcommand{\RN}{\mathbb{R}^{N}}
\newcommand{\Rn}{\mathbb{R}^{n}}
  
\newcommand{\N}{\mathbb{N}}

\newcommand{\ep}{\varepsilon}
\newcommand{\h}{\widehat}

\newcommand{\M}{\mathbb{R}^{N \times m}}
\newcommand{\nab}{\nabla}
\newcommand{\wc}{\rightharpoonup}
\newcommand{\mbf}{\mathbf}
\newcommand{\n}{\nab_{\mbf a}}
\newcommand{\p}{\partial}
\newcommand{\LL}{\mathrm{L}}
\newcommand{\WW}{\mathrm{W}}
\newcommand{\CC}{\mathrm{C}}
\newcommand{\Leb}{\mathscr{L}}
\newcommand{\dd}{\mathrm{d}}
\newcommand{\loc}{\mathrm{loc}}
\newcommand{\HH}{\mathcal{H}}

\newtheorem{theorem}{Theorem}[section] 
\newtheorem{lemma}[theorem]{Lemma}
\newtheorem{prop}[theorem]{Proposition}   
\newtheorem{corollary}[theorem]{Corollary}

\theoremstyle{definition}
\newtheorem{definition}[theorem]{Definition}

\DeclareMathOperator{\Reg}{Reg}
\DeclareMathOperator{\Sing}{Sing}

\title[Mixed-smoothness regularity]{Regularity for minimizers of mixed-smoothness variational integrals}

\author[J.~Kristensen]{Jan Kristensen}
\author[A.~Prosinski]{Adam Prosinski}
\author[A.~Strong]{Aidan Strong}

\address{Mathematical Institute, University of Oxford\\Andrew Wiles Building, Radcliffe Observatory Quarter, Woodstock Road, Oxford, OX2 6GG, Oxford, United Kingdom}

\email{kristens@maths.ox.ac.uk}

\email{prosinski.adam@gmail.com}

\email{aidan.strong@maths.ox.ac.uk}

\begin{document}
\baselineskip=15pt

\begin{abstract}
We establish the first partial regularity results for minimizers of a class of suitably quasiconvex variational integrals defined on anisotropic Sobolev spaces.
These variational integrals are related to higher gradient continuum models for pantographic metamaterials.
\end{abstract}

\keywords{Quasiconvexity, partial regularity, mixed smoothness, anisotropic Sobolev space}

\maketitle

\begin{section}{Introduction}
Over the past 30 years the modelling and design of novel, exotic metamaterials with desirable properties have led
researchers to consider a number of higher gradient continuum models for pantographic structures \cite{metamaterials,dellisola2015designing,stilz2022continuum}.
An important first step clarifying the mathematical basis of such theories concerns linear pantographic sheets
\cite{eremeyev2018linear} whose stored energy densities are quadratic forms in certain higher order derivatives.
One of the simplest such models is the quadratic functional
\begin{equation}\label{quadraticcase}
\mathfrak{G}[v]=\int_{\Omega} \! \Bigl( |\p_1v|^2+|\p_{2}^2v|^2 \Bigr) \, \dd x
\end{equation}
defined on suitable maps $v \colon \Omega \subset \R^2 \to \R^2$, where we note that the different differentiation orders reflect
the desired feature that the metamaterial stores different amounts of energy in different directions. The authors of \cite{eremeyev2018linear}
then proceeded to establish existence and uniqueness of weak solutions to related systems and of minimizers in suitable classes of anisotropic Sobolev spaces.
While \cite{eremeyev2018linear} did not address the issue of regularity of weak solutions, it is not difficult to show that minimizers of the particular functional
$\mathfrak{G}$ are all $\CC^\infty$ smooth despite lack of ellipticity in the usual Legendre-Hadamard sense. Indeed most of the functionals
considered in \cite{eremeyev2018linear} are semi-elliptic in the sense of H\"{o}rmander \cite{hormander1969linear} (the definition is recalled in Section 3) and
as such they are hypoelliptic too and therefore their interior regularity follows by classical means.
As the title of \cite{eremeyev2018linear} suggests their aim is to develop a linear theory and with the above remarks concerning regularity theory in this
setting, together with the material presented in Section 3 below, the linear theory has a reasonably firm foundation.
Now one of the desired properties of these metamaterials is that their elastic range should be many
orders of magnitude wider than that exhibited by more conventional materials. It then becomes a natural task to extend the mathematical
framework considered in \cite{eremeyev2018linear} to the more general nonlinear case. Our aim here is to establish partial regularity for
minimizers for a class of variational problems that, at least at a phenomenological level, is a first step towards a
nonlinear variational framework for pantographic metamaterials. Before turning to our results we pause to briefly describe the set-up, where
the main point is to allow for partial derivatives of different orders and otherwise be guided by the classical variational theory.
For a fixed multi-index with non-zero entries $\mbf{a}=(a_1, \dots , a_n) \in \mathbb{N}^n$, and a given map $v\colon \Omega \rightarrow \RN$ defined and locally
integrable on an open and bounded subset $\Omega$ of $\Rn$ we define its $\mbf{a}$-gradient by
$$
\n v(x) := \bigl\{ \p^\alpha v(x) \bigr\}_{\langle \alpha, \mbf{a}^{-1}\rangle=1},
$$
where the partial derivatives are understood in the sense of distributions, $\mbf{a}^{-1}=(1/a_1, \dots, 1/a_n)$ and
$\langle \alpha, \mbf{a}^{-1}\rangle = \sum_{i=1}^{n} \alpha_{i}/a_{i}$. Vectors in $v \in \RN$ are identified with column $N \times 1$ matrices and so we
note in particular that $\partial^{\alpha}v$ is the column matrix whose entries are the partial derivatives of the coordinate functions. The partial derivatives
that appear in the $\mbf{a}$-gradient are exactly those corresponding to multi-indices that belong to the set
$\mathcal{H}=\left\{\alpha: \langle \alpha, \mbf{a}^{-1}\rangle=1 \right\}$ that we, following \cite{kazstowoj17}, call a pattern of homogeneity.
Let $m$ be the number of elements in $\mathcal{H}$, order the multi-indices lexicographically
and consider the $\mbf{a}$-gradient as an $N \times m$ matrix with real entries. Working with such collections of derivatives allows us to
use scaling arguments known from the standard set-up (\cite{Dacorogna,Rindler} and \cite{dibenedetto,GG}); moreover a corresponding theory for anisotropic
Sobolev spaces, $\WW^{\mbf{a},p}$, has been developed by Nikolskii and coworkers (we recall the relevant definitions and facts in Section 2).
We focus here on variational integrals of the basic form
\begin{equation}\label{eqn mixed functional}
  \mathfrak{F}[v]=\int_{\Omega} \! F(\n v) \, \dd x,
\end{equation}
where $F \colon \M \rightarrow \R$ is a continuous function, henceforth called an integrand, $\Omega$ is a bounded open subset of $\Rn$ and
$v \colon \Omega \to \RN$ are suitably smooth maps. When the integrand $F$ satisfies a $p$-growth condition the integral is well-defined
as a Lebesgue integral for maps $v$ whose $\mbf{a}$-gradients are in $\LL^p$, a property that is ensured by membership of the $\WW^{\mbf{a},p}$
mixed-smoothness Sobolev space. Combining our arguments with methods known from the standard variational set-up \cite{GG} it is possible to also establish
partial regularity results for minimizers of more general variational integrals. This extension is however rather technical and not entirely straightforward
and we have therefore decided to focus on the basic case \eqref{eqn mixed functional} here and leave the general case to future developments.

A systematic study of variational problems in this mixed-smoothness setting was undertaken by the second author in his DPhil thesis \cite{Pros,Pros2}.
There, the notion of $\mbf{a}$-quasiconvexity was studied to facilitate use of the direct method to prove existence of minimizers.
Accordingly, a continuous integrand $F\colon \M \rightarrow \R $ is $\mbf{a}$-quasiconvex at $z_{0} \in \M$ if
\begin{equation}\label{def quasiconvexity}
  \int_{\Rn} \! \Bigl( F(z_{0}+\n \phi (x))-F(z_{0}) \Bigr) \, \dd x \geq 0
\end{equation}
holds for all $\phi \in \CC_{c}^{\infty}(\Rn ,\RN )$. If $F$ is $\mbf{a}$-quasiconvex at all $z_{0} \in \M$ we simply say that $F$ is $\mbf{a}$-quasiconvex.
This definition was motivated by observations made in \cite{kazstowoj17} who introduced $\mbf{a}$-quasiconvexity in connection to Ornstein non-inequalities
in a mixed smoothness set-up.
The reader will notice that $\mbf{a}$-quasiconvexity is a natural extension to the mixed-smoothness set-up of Morrey's quasiconvexity \cite{morreyQC} and that it is distinct from
the notions of generalized quasiconvexity previously considered in the literature, including \cite{meyers,fonmul,iwsb}. As we will show in Section 2,
this difference is particularly pronounced when the pattern of homogeneity $\mathcal{H}$ contains multi-indices whose orders have different remainders modulo $4$.
In such cases the linearization of $\mbf{a}$-quasiconvexity, that we call $\mbf{a}$-rank-one convexity, is a sort of plurisubharmonicity. There is then also
the added technical issue that the usual local Lipschitz bounds \cite{Dacorogna,Rindler} known for rank-one convexity are
not available. However, when the orders of multi-indices in the pattern of homogeneity $\mathcal{H}$ all are congruent modulo $4$, then $\mbf{a}$-rank-one
convexity is simply convexity in the directions of a spanning cone and then the local Lipschitz bounds for $\mbf{a}$-rank-one convex integrands follow in the
usual way. In such cases the difference to previous set-ups is then merely about the intrinsic anisotropic scaling.
However, despite these issues the notion of $\mbf{a}$-quasiconvexity plays at the abstract level exactly the same role
as Morrey's quasiconvexity in the classical set-up: it is responsible for both coercivity and lower semicontinuity \cite{Pros,Pros2}.
More precisely, it was shown in \cite[Theorem 6.1]{Pros2} that for nonnegative, continuous integrands $F$ satisfying a $p$-growth assumption,
$0 \leq F(z)\leq L\bigl( |z|^{p}+1 \bigr)$, sequential weak lower semicontinuity of $\mathfrak{F}$ defined at \eqref{eqn mixed functional} on $\WW^{\mbf{a},p}(\Omega , \RN )$
is equivalent to $\mbf{a}$-quasiconvexity of $F$. In fact, under this growth and nonnegativity condition, it is equivalent to the stronger notion
of closed $\WW^{\mbf{a},p}$-quasiconvexity of $F$ and so is in turn equivalent to lower semicontinuity in a compensated compactness set-up (see \cite{Tartar79},
\cite{jk2015} and \cite{Pros,Pros2}).

Coercivity was characterized in \cite[Theorem 4.3]{Pros2}, where it in particular was shown that for a continuous
integrand satisfying the $p$-growth condition, $|F(z)| \leq L\bigl( |z|^{p}+1 \bigr)$, coercivity of $\mathfrak{F}$ defined at \eqref{eqn mixed functional} on
$\WW^{\mbf{a},p}$-Dirichlet classes is equivalent to existence of $\ell > 0$ and $z_{0} \in \M$ such that the integrand $z \mapsto F(z) -\ell |z|^{p}$
is $\mbf{a}$-quasiconvex at $z_0$. In order to convince the reader that our main result, stated as Theorem \ref{ThmPartialregularity} below, is given under the
natural assumptions we state the following existence result, where we formulate the coercivity condition in terms of the reference integrand
\begin{equation}\label{referenceintegrand}
 V_{p}(z) := \Bigl( 1+|z|^{2} \Bigr)^{\frac{p}{2}}-1.
\end{equation}
We remark in this connection that the coercivity on Dirichlet classes is equivalent to $\mbf{a}$-quasiconvexity of the difference $F-G$ at some $z_0$ for any choice of
integrand $G$ satisfying 0 < $\liminf_{|z| \to \infty}G(z)/|z|^{p} \leq \limsup_{|z| \to \infty}G(z)/|z|^{p} < +\infty$. The choice of $G=V_{p}$ is simply
  for easy comparison with our regularity result stated later.
\begin{prop}\label{existence}
Let $F \colon \M \to \R$ be an $\mbf{a}$-quasiconvex integrand of class $\CC^1$ satisfying for some exponent $p \in (1,\infty )$ and constant $L>0$ the
local Lipschitz condition $\bigl| F^{\prime}(z) \bigr| \leq L\bigl( |z|^{p-1}+1 \bigr)$ for all $z \in \M$. Assume that for some $\ell > 0$ the integrand $z \mapsto F(z) -\ell V_p$
is $\mbf{a}$-quasiconvex at some $z_{0} \in \M$.

Then for any bounded open subset $\Omega$ of $\Rn$ and any $g \in \WW^{\mbf{a},p}_{\mathrm{loc}}(\Rn , \RN )$ the variational integral
$$
\mathfrak{F}[v] = \int_{\Omega} \! F\bigl( \nabla_{\mbf{a}}v(x) \bigr) \, \dd x
$$
admits a minimizer on the Dirichlet class $\WW^{\mbf{a},p}_{g}( \Omega , \RN ) = g+\WW^{\mbf{a},p}_{0}( \Omega , \RN )$. Furthermore, any minimizer
$u \in \WW^{\mbf{a},p}_{g}( \Omega , \RN )$ satisfies the Euler-Lagrange system in the weak sense:
$$
\int_{\Omega} \! F^{\prime}\bigl( \nabla_{\mbf{a}}u(x) \bigr) \cdot \nabla_{\mbf{a}}\phi (x) \, \dd x = 0 \quad \forall \phi \in \WW^{\mbf{a},p}_{0}( \Omega , \RN ) .
$$
\end{prop}
We omit the proof here as it follows by standard means and the results of \cite{ChJK2017}, \cite{Pros2}, some of which are recalled below in Section 2 and needed later in our
regularity proof. In fact, it is possible to obtain natural existence results also under less restrictive growth conditions, as in \cite{jk2015}, \cite{ChJK2017}, \cite{gk24},
and we intend to report on that elsewhere. However, we emphasize that besides the assumption of $p$-growth all other assumptions are in terms of $\mbf{a}$-quasiconvexity
and that these assumptions are sharp for the requisite sequential weak lower semicontinuity and coercivity in the $\WW^{\mbf{a},p}$-Dirichlet class.

This settles the question of existence of minimizers in Dirichlet classes, at least at the abstract level. The reason for this qualification is that, just like
in the standard case, $\mbf{a}$-quasiconvexity is not easy to establish in general. Consider the functional $\mathfrak{G}$ defined at \eqref{quadraticcase}: it fits
into the above mixed smoothness set-up with $\mbf{a}=(1,2)$. Clearly the functional $\mathfrak{G}$ has a convex integrand and so the functional is of course
convex too. It is not difficult to show that any quadratic $\mbf{a}$-quasiconvex integrand defines a functional that is convex on $\WW^{\mbf{a},2}$-Dirichlet 
classes despite the fact that a quadratic $\mbf{a}$-quasiconvex integrand need not itself be convex as a function on $\M$. In fact, since $\mbf{a}$-quasiconvexity
generalizes the usual notion of quasiconvexity in the sense of Morrey, we should not expect the variational integral corresponding to an $\mbf{a}$-quasiconvex
integrand to have the property that it is convex on $\WW^{\mbf{a},p}$-Dirichlet classes. We therefore again have the phenomenon that the considered variational
integrals may admit extremals that are not minimizers and also various classes of local minimizers that are not (absolute) minimizers and that (absolute) minimizers in
$\WW^{\mbf{a},p}$-Dirichlet classes might be non-unique (as is well-known in the classical set-up \cite{MullerSverak03,krta03,sz04,sp09}). 

The task we set ourselves here is to establish partial regularity of minimizers in the spirit of the celebrated result of Evans \cite{Evans1986Quasiconvexity}
(see also \cite{AcerbiFusco1987,fuhu85,gimo86} and \cite{GG,mi06} and the references therein). We recall that well-known examples \cite{sverakyan,MooneySavin16}
show that one cannot expect minimizers to be fully regular even under very strong convexity assumptions and that the study of partial regularity goes back
to \cite{GiustiMiranda68,Morrey68}.

In order to prove partial regularity one typically must strengthen the quasiconvexity assumption \eqref{def quasiconvexity}.
Thus, for functionals of the form (\ref{eqn mixed functional}) we assume that the integrand $F \colon \M \to \R$ is of class $\CC^2$ and that it satisfies:

\begin{enumerate}[(H1)]\label{assumptions on F}
    \item $|F(z)| \leq L(1+|z|^p)$
    \item $|F'(z)| \leq L(1+|z|^{p-1})$
    \item For every $M>0$ there exists $l_M >0$ such that $F-l_MV_p$ is $\mbf{a}$-quasiconvex at $z_0$ for all $|z_0| \leq M$. 
\end{enumerate}
Recall that here, $V_p$, denotes the reference integrand defined at \eqref{referenceintegrand}. We remark that in the considered generality, where the
linearization of $(H3)$ amounts to a sort of plurisubharmonicity, the Lipschitz condition $(H2)$ is needed and is not a consequence of the $p$-growth condition $(H1)$. 
As already commented above, the hypotheses $(H1)$, $(H2)$, $(H3)$ in particular facilitate use of the direct method to prove existence of minimizers of $\mathfrak{F}$ in
$\WW^{\mbf{a},p}$-Dirichlet classes. These minimizers are especially local minimizers, meaning maps $u \in \WW_{\loc}^{\mbf{a},p}(\Omega , \RN)$ that for
each $\Omega^{\prime} \Subset \Omega$ satisfy
$$
\int_{\Omega^{\prime}} \! F( \n u) \, \dd x \leq \int_{\Omega^{\prime}} \! F( \n u + \n \phi ) \, \dd x
$$
for all $\phi \in \WW^{\mbf{a},p}_{0}( \Omega^{\prime}, \RN )$.
Our main theorem concerns regularity of such local minimizers:

\begin{theorem}\label{ThmPartialregularity}
Let $2 \leq p < \infty$, and let $u \in \WW_{\loc}^{\mbf{a},p}(\Omega , \RN )$ be a local minimizer of $\mathfrak{F}$ under the assumptions $(H1), (H2), (H3)$.
Let
$$
\Reg(u)=\left\{x \in \Omega: \, \n u \text{ is } \CC^{0,\theta}_{\delta_{\mbf{a}}} \text{ for all }0<\theta<1  \text{ in a neighbourhood of }  x\right\}
$$
and $\Sing(u)=\Omega \setminus \Reg(u)$. Then the open subset $\Reg(u)$ has full measure in $\Omega$, that is, $\Leb^{n} \bigl( \Omega \setminus \Reg(u) \bigr) =0$.
In particular, $\n u$ is Hölder continuous on $\Reg(u)$ in direction $x_j$ with exponent $\theta a_{j}/\max a_{i}$ for all $0<\theta<1$, $j=1, \dots, n$.
Moreover, the singular set is characterised by $\Sing(u)=\Sing_{1}(u) \cup \Sing_{2}(u)$, where
\begin{equation}\label{eqnSing1}
\Sing_{1}(u)=\left\{ x_{0}: \, \liminf_{R \searrow 0} \strokedint_{Q_R(x_0)}V_p(\n u-(\n u)_{Q_R(x_0)}) \dd x >0\right\}
\end{equation}
\begin{equation}\label{eqnSing2}
\Sing_{2}(u)=\left\{x_{0}: \, \limsup_{R \searrow 0} |(\n u)_{Q_R(x_0)}| =\infty \right\}
\end{equation}
\end{theorem}
Here, $\delta_\mbf{a}$ and $Q_R$ are the anisotropic metric and boxes defined in Section 2.
Our proof follows a standard strategy that goes back to works by Almgren and De Giorgi on regularity of minimal surfaces. It relies on a suitable Caccioppoli
inequality and a harmonic approximation argument, where we for the latter follow the direct method of \cite{gk19,gk24}. To carry this out, we need suitable
linear estimates in the mixed smoothness setting which we collect in Section 3.

\subsection{Notation}
Our notation is mostly standard, but we mention some conventions. We write $Q_{R}(x_0)$ for the anisotropic box of radius $R$ centred at $x_0$ (see Section 2).
For any bounded measurable set $D \subset \Rn$ with Lebesgue measure $\Leb^{n}(D) = |D|>0$, we write $(v)_D=\strokedint_{D}v \dd x = \frac{1}{|D|}\int_{D}v \dd x$ to denote
the integral mean of a function $v \in \LL^{1}(D)$ on $D$ (and similar notation is used for vector-valued maps). In the case of anisotropic boxes we often
write $(v)_R := (v)_{Q_R(x_0)}$.
We will say a constant $C>1$ is universal if it only depends on $n, N,p,\mbf{a}$, and in this case we will just write $C>1$ which may vary from line to line.
We write $A \approx B$ if there exists some universal constant $C>1$ such that $C^{-1}A\leq B \leq CA$. We also write $A \lesssim_M B$ if $A \leq C_M B$,
where $C>1$ is some constant that may depend on $M$. We write $\langle \cdot, \cdot \rangle$ for the Frobenius inner product acting on $\N_{0}^n$, while we use
$z\cdot w = \mbox{Trace }(z^{T}w)$ and $|z|^{2} = z \cdot z$ to denote the usual Frobenius inner product and norm on matrices $z$, $w \in \M$.
For $\mbf{a} \in \N^n$ we write $\mbf{a}^{-1}=(1/a_1, \dots, 1/a_n)$ and $|\mbf{a}^{-1}|=\sum_{i=1}^{n} 1/a_{i}$.
\end{section}

\begin{section}{Preliminary results}
In this section we collect requisite results concerning Sobolev spaces of mixed smoothness and $\mbf{a}$-quasiconvex integrands. While all results should
be considered known we provide brief proofs when we couldn't find a precise reference in the literature.
\subsection{Anisotropic Sobolev spaces and spaces of mixed smoothness}
Such spaces have been studied extensively, and our main references are \cite{Besovintrep} and \cite{Demirussian}.
\begin{definition}\label{def:Wa,p}
Let $\Omega \subset \Rn$ be an open set, $1\leq p < \infty.$ Let $\mbf{a}=(a_1,\dots,a_n) \in \N^n$. Let $\mbf{a}^{-1}=(1/a_1,\dots,1/a_n)$.
We define the anisotropic Sobolev space
$$
\WW^{\mbf a,p}(\Omega , \RN )=\left\{u \in \LL^{p}(\Omega , \RN ): \, \partial^{\beta}u \in \LL^{p}(\Omega , \RN ) \text{ for all } \langle \beta, \mbf{a}^{-1} \rangle \leq 1 \right\}
$$
equipped with the norm
$$
\Vert u\Vert_{\WW^{\mbf a,p}(\Omega , \RN )}=\sum_{\langle \beta, \mbf{a}^{-1}\rangle \leq 1} \Vert \partial^{\beta}u \Vert_{\LL^{p}(\Omega , \RN )}.
$$
We define $\WW_0^{{\mbf a},p}(\Omega , \RN )$ as the completion of $\CC_c^{\infty}(\Omega ,\RN)$ with respect to the norm $\Vert \cdot \Vert_{\WW^{\mbf{a},p}(\Omega ,\RN)}$.
\end{definition}
Note that $a_i$ is the highest-order  derivative considered in the $e_i$-direction.
For notational simplicity, we will sometimes omit the explicit dependence on the target space and simply write $\WW^{\mbf{a},p}(\Omega)$.
We  denote the admissible set of derivatives by $\mathcal{A} = \left\{\beta \in \N_{0}^{n}: \, \langle \beta, \mbf{a}^{-1} \rangle \leq 1  \right\}$, and the derivatives lying
on the hyperplane by $\mathcal{H}=\left\{\alpha \in \N_{0}^{n}: \, \langle \alpha, \mbf{a}^{-1} \rangle =1 \right\}.$
We define the anisotropic gradient $\n u = (\p^\alpha u)_{\alpha \in \mathcal{H}}$, that is the matrix whose $i$-th row consists of all derivatives $\p^\alpha u^i$ for
$\alpha \in \mathcal{H}$, ordered lexicographically. 
When studying mixed smoothness Sobolev spaces, it is natural to restrict attention to domains having the following geometric property (see \cite[Sect.~8]{Demi}).
\begin{definition}\label{def:horn}
An $\mbf{a}$-horn is a domain of the form 
$$
\Lambda(r,\mbf{t},\mbf{s})=\bigcup_{0<\rho<r}\left\{y \in \R^n: t_j<\frac{y_j}{\rho^{1/a_j}}<s_j \text{ for all }  j=1, \dots, n\right\}
$$
where $r>0$, $0<t_j<s_j$ for all $j=1, \dots, n$.
\begin{enumerate}[(i)]
\item A subset $\omega \subset \Omega$ satisfies the $\mbf{a}$-horn condition relative to $\Omega$ if there exists an $\mbf{a}$-horn $\Lambda$ such that
  $x+\Lambda \subset \Omega $ for every $x\in \omega$.
\item A subset $\Omega \subset \R^n$ satisfies the weak $\mbf{a}$-horn condition if we can write
$$
\Omega = \bigcup_{i=1}^k \omega_i
$$
where each $\omega_i$, possibly after a coordinate reflection, satisfies the $\mbf{a}$-horn condition relative to $\Omega$.
\item A subset $\Omega \subset \Rn$ satisfies the strong $\mbf{a}$-horn condition if it satisfies the weak $\mbf{a}$-horn condition with $\omega_i$ as above,
and moreover there exists $\delta>0$ such that
$$
\Omega = \bigcup_{i=1}^k \omega_i^{\delta}
$$
where $\omega_i^{\delta}=\left\{x \in \omega_{i}: \, \text{dist}(x, \Omega \setminus \omega_i)>\delta \right\}.$
 \end{enumerate}
\end{definition}
The motivation behind this definition is that embedding theorems are typically proven via integral representations, and such representations involve integration
over $\mbf{a}$-horn regions. In practice, we will typically work with rectangular domains of the form $(l_1,r_1)\times \dots \times (l_n,r_n)$ where
$l_i, r_i \in [-\infty, \infty]$, and such domains are easily seen to satisfy the strong $\mbf{a}$-horn condition (see \cite{Besovintrep}).
Note that if $\mbf{a}=(k, \dots, k)$ is isotropic then $\mbf{a}$-horns are a special class of capped cones. 

\begin{prop}\label{prp:Wa,p}
Let $\Omega \subset \Rn$ be an open set satisfying the weak $\mbf{a}$-horn condition, and let $1<p<\infty$. 

\begin{enumerate}[(i)]
\item $\WW^{{\mbf a},p}(\Omega)$ is a separable reflexive Banach space.
\item $\Vert u \Vert_{\WW^{{\mbf a},p}(\Omega)} \approx \Vert u \Vert_{\LL^{p}(\Omega)}+\sum_{i=1}^n \Vert \partial^{a_i}_iu\Vert_{\LL^{p}(\Omega)}$ defines an equivalent norm on
  $\WW^{\mbf{a},p}(\Omega)$, where the implied constants depend on $\Omega, n, p, \mbf{a}$.
\item $\CC^\infty(\Omega) \cap \WW^{\mbf{a},p}(\Omega)$ is dense in $\WW^{{\mbf a},p}(\Omega)$.
\item Assume $\Omega$ is bounded with Lipschitz boundary. If $u_j \wc u$ weakly in $\WW^{{\mbf a},p}(\Omega)$ then $\partial^\beta u_j \rightarrow \partial^\beta u$
  strongly in $\LL^p(\Omega)$ for every $\langle \beta, {\mbf a}^{-1}\rangle<1$. 
\item If $\Omega$ satisfies the strong $\mbf{a}$-horn condition, then there exists a bounded linear extension operator
  $\WW^{{\mbf{a}},p}(\Omega) \rightarrow \WW^{{\mbf{a}},p}(\Rn ).$
\end{enumerate}
\end{prop}
We refer to \cite{Pros2} for a more detailed survey of the above results. 
\begin{definition}\label{def:box}
An anisotropic box is a set of the form
$$
Q_R = Q_R(x_0) := \bigl\{x \in \Rn : |x_i-(x_0)_i|^{a_i} <R \text{ for all } 1\leq i \leq n \bigr\}
$$
\end{definition}
For scaling reasons, anisotropic boxes will replace usual cubes (or balls) in all of our estimates.
Note that anisotropic boxes satisfy the strong $\mbf{a}$-horn condition. We can define the  anisotropic metric on $\Rn$ by

\begin{equation}\label{anisotropic metric}
  \delta_{\mbf{a}}(x,y)= \sup_{1\leq j \leq n}|x_j-y_j|^{a_j/{\max a_i}}.
\end{equation}
If we denote by $B^{\mbf{a}}_R(x_0) \subset \Rn$ the open ball of radius $R$ with centre $x_0$ with respect to the anisotropic metric, then it is easy to see that
$Q_R(x_0)=B^{\mbf{a}}_{R^{1/\max a_i}}(x_0)$. We observe that $|Q_R|=2^nR^{|\mbf{a}^{-1}|}$, from which it follows $(\Rn , \delta_\mbf{a}, \Leb^{n} )$ is a
doubling metric measure space. Consequently the Lebesgue differentiation theorem holds for anisotropic boxes, in the sense that for exponents $p \in [1,\infty )$ and 
any representative of $f \in \LL^{p}_{\loc}(\Rn )$ we have
\begin{equation}
 \lim_{R \searrow 0}\strokedint_{Q_R(x_0)} \! |f(x)-f(x_0)|^{p} \, \dd x=0
\end{equation}
for almost every $x_0 \in \Rn$.

Working with anisotropic boxes, we can state the following scaled version of the Sobolev embedding theorem (see \cite{Besovintrep}).

\begin{prop}\label{prp:sobolev embedding}
Let $u \in \WW^{{\mbf a},p}(Q_R)$ and $\langle \beta, {\mbf a}^{-1} \rangle<1$. For $1<p<q<\infty$, write
$\chi_{\beta} :=\langle \beta, {\mbf a}^{-1} \rangle+(1/p - 1/q)|{\mbf a}^{-1}|$. If $\chi_\beta \leq 1$ then
$\partial ^\beta u \in \LL^{q}(Q_R)$ and there exists $C=C(n,N, {\mbf a}, \beta, p,q)$ such that
\begin{equation}
\Vert \partial^\beta u \Vert_{\LL^{q}(Q_R)} \leq CR^{1-\chi_\beta} \sum_{i=1}^n \Vert \partial^{a_i}_iu \Vert_{\LL^{p}(Q_R)} + CR^{-\chi_\beta}\Vert u \Vert_{\LL^{p}(Q_R)}.
\end{equation}
\end{prop}
Let us remark briefly how one obtains the scaled version, since scaling arguments will be implicitly used throughout the paper.
Indeed, assume the inequality holds in the case $R=1$ and $x_{0}=0$. Then for $u \in \WW^{\mbf{a},p}(Q_{R}(x_{0}))$ we set $v(x)=u(x_0+R \odot x)$ where
$R \odot x = (R^{1/a_1 }x_1,\dots, R^{1/a_n}x_n)$. Then $v \in \WW^{\mbf{a},p}(Q_{1}(0))$, so we apply the inequality to $v$.
Since we can relate the norms of $v$ and $u$ by
$$
\int_{Q_{1}(0)} \! |\p^\beta v|^{p} \, \dd x =R^{p\langle \beta,\mbf{a}^{-1} \rangle-|\mbf{a}^{-1}|} \int_{Q_{R}(x_{0})} \! |\p^\beta u|^{p} \, \dd x
$$
the inequality for $u$ readily follows. 

We will also need a version of the Poincaré inequality in the mixed smoothness setting.
In the isotropic case  $\mbf{a}=(m, \dots, m)$, this is often also referred to as the Bramble--Hilbert lemma.

\begin{prop}\label{prp:poincare with perturbation}
Let $1 <p<\infty$ and write $Q_{R}=Q_{R}(x_{0})$.
There exists $C=C(n,N,\mbf{a},p)$ such that for any $u \in \WW^{\mbf{a},p}(Q_R)$ and any constant $z_0 \in \M$, there
exists some polynomial $P$ with $\n P=z_0$ such that
$$
\Vert \partial^\beta( u-P) \Vert_{\LL^{p}(Q_R)} \leq CR^{1-\langle \beta, {{\mbf a}}^{-1} \rangle } \Vert \nab_{{\mbf a}}u-z_0 \Vert_{\LL^{p}(Q_R)}
$$ 
for all $\langle \beta, {{\mbf a}}^{-1}\rangle \leq 1.$ 
\end{prop}
A constructive proof (for $x_{0}=0$, $R=1$) follows from \cite[Theorem 4.2]{Dupnot}, where $P$ is constructed in such a way that
$$
(u-P_u)(x)=\sum_{\alpha \in \mathcal{H}} \int_{Q_{1}(0)} \! k_\alpha(x,y) \p^\alpha u(y) \, \dd y
$$
where $|k_\alpha(x,y)| \leq C|x-y|^{|\alpha|-n}$ and $k_\alpha(x, \cdot)$ has compact support in $Q_{1}(0)$ for all $x \in Q_{1}(0)$.
Thus, Proposition \ref{prp:poincare with perturbation} follows from the $\LL^p$ boundedness of the Riesz potentials and Proposition \ref{prp:Wa,p}(ii).
Finally, we define Campanato spaces adapted to our scaling, which for simplicity we only need when the domain is an anisotropic box $Q$.
\begin{definition}\label{def: campanato}
Fix an anisotropic box $Q$ and let $I(x_0,r)=Q_r(x_0) \cap Q$. For $\lambda \geq 0$ we define the anisotropic Campanato space 
$$
\mathfrak{L}^{p,\lambda}_{\mbf{a}}(Q)=\left\{u \in \LL^{p}(Q): \, \sup_{x_0 \in Q, r>0} |I(x_0,r)|^{-\lambda } \int_{I(x_0,r)} \! |u-(u)_{I(x_0,r)}|^{p} \, \dd x < \infty\right\}.
$$
\end{definition}
The following embedding can be found in \cite[Definition 2.4]{GiustiItalian}.

\begin{lemma}\label{lem: campanato}
If $\lambda >1$ then $\mathfrak{L}^{p,\lambda}_{\mbf{a}}(Q) \cong \CC^{0, \chi}_{\delta_{\mbf{a}}}(Q)$ with $\chi = \frac{\max a_i |\mbf{a}^{-1}|}{p}(\lambda-1)$.
In particular, functions $u \in \mathfrak{L}_\mbf{a}^{p, \lambda}(Q)$ are Hölder continuous in direction $x_j$ with exponent $\frac{a_j}{\max a_i}\chi$. 
\end{lemma}

\subsection{$\mbf{a}$-quasiconvexity}
The notion was introduced in \cite{kazstowoj17} and further developed in \cite{Pros,Pros2}.
\begin{definition}
A continuous integrand $F \colon \M \to \R$ is $\mbf{a}$-quasiconvex at $z_0 \in \M$ if
$$
F(z_0) \leq \strokedint_{Q_{1}(0)} \! F(z_0+\n \phi) \, \dd x
$$
for all $\phi \in \CC_c^{\infty}(Q_{1}(0) , \RN )$. 
We say $F$ is $\mbf{a}$-quasiconvex if it is $\mbf{a}$-quasiconvex at all $z_0 \in \M$. 
\end{definition}
As in the case of quasiconvexity in the sense of Morrey, the definition is independent of the choice of domain $Q_{1}(0)$ \cite{Pros2}.


The following proposition will be useful later in section 3. In its statement $( \mbf{e}_{\alpha} )_{\alpha \in \mathcal{H}}$ denotes the canonical basis for $\R^m$,
where we recall that the pattern of homogeneity, $\HH$, has $m$ elements and is ordered lexicographically.

\begin{prop}\label{propLegendreHadamard}
Suppose that $F \colon \M \to \R$ is $\CC^2$ and $\mbf{a}$-quasiconvex at $z \in \M$. Then 
\begin{eqnarray}
0 &\leq& F^{\prime \prime}(z)\bigl[ v \otimes \sum_{\alpha \in \mathcal{H}_{0}} \xi^{\alpha}\mbf{e}_{\alpha}, v \otimes \sum_{\alpha \in \mathcal{H}_{0}} \xi^{\alpha}\mbf{e}_{\alpha} \bigr] +
F^{\prime \prime}(z)\bigl[ v \otimes \sum_{\alpha \in \mathcal{H}_{2}} \xi^{\alpha}\mbf{e}_{\alpha}, v \otimes \sum_{\alpha \in \mathcal{H}_{2}} \xi^{\alpha}\mbf{e}_{\alpha} \bigr]\label{arc}\\
&& -2\sum_{\alpha \in \HH_{0}, \beta \in \HH_{2}} F^{\prime \prime}(z)\bigl[ v \otimes \xi^{\alpha}\mbf{e}_{\alpha}, v\otimes \xi^{\beta}\mbf{e}_{\beta} \bigr]\nonumber\\
&& +F^{\prime \prime}(z)\bigl[ v \otimes \sum_{\alpha \in \mathcal{H}_{1}} \xi^{\alpha}\mbf{e}_{\alpha}, v \otimes \sum_{\alpha \in \mathcal{H}_{1}} \xi^{\alpha}\mbf{e}_{\alpha} \bigr] +
F^{\prime \prime}(z)\bigl[ v \otimes \sum_{\alpha \in \mathcal{H}_{3}} \xi^{\alpha}\mbf{e}_{\alpha}, v \otimes \sum_{\alpha \in \mathcal{H}_{3}} \xi^{\alpha}\mbf{e}_{\alpha} \bigr]\nonumber\\
&& -2\sum_{\alpha \in \HH_{1}, \beta \in \HH_{3}} F^{\prime \prime}(z)\bigl[ v \otimes \xi^{\alpha}\mbf{e}_{\alpha}, v\otimes \xi^{\beta}\mbf{e}_{\beta} \bigr]\nonumber
\end{eqnarray}
holds for all $v \in \RN$ and $\xi \in \Rn$, where $\mathcal{H}_{j} = \bigl\{ \alpha \in \mathcal{H} : \, | \alpha | \equiv j \mbox{ mod } 4 \bigr\}$ for $j=0,1,2,3$,
and we use the convention that an empty sum is $0$.
\end{prop}
We call a $\CC^2$ integrand $F$ satisfying the condition \eqref{arc} at all $z \in \M$ for $\mbf{a}$-rank-one convex. Note that when all multi-indices
in the pattern of homogeneity $\HH$ have the same remainder modulo $4$, say $\HH = \HH_j$, then in \eqref{arc} only the term with $\HH_j$  appears on the
right-hand side and so $\mbf{a}$-rank-one convexity is in this case simply convexity in the directions of $v \otimes \sum_{\alpha \in \mathcal{H}} \xi^{\alpha}\mbf{e}_{\alpha}$,
where $v \in \RN$, $\xi \in \Rn$. In all other cases the notion of $\mbf{a}$-rank-one convexity is more complicated and is a sort of plurisubharmonicity.
Proposition \ref{propLegendreHadamard} rectifies an issue in \cite[Theorem 9]{kazstowoj17} and \cite{Pros,Pros2}.
Our proof is essentially Van Hove's \cite{vanhove} proof that, for quadratic forms, rank-one convexity implies quasiconvexity, run in reverse.
It is convenient to denote
$\mathcal{H}_e = \bigl\{ \alpha \in \mathcal{H} : \, | \alpha | \mbox{ is even} \bigr\} = \HH_{0} \cup \HH_{2}$ and
$\mathcal{H}_o = \bigl\{ \alpha \in \mathcal{H} : \, | \alpha | \mbox{ is odd} \bigr\} = \HH_{1} \cup \HH_{3}$.

\begin{proof}
Let $\phi \in \CC^{\infty}_{c}( \Rn ,\RN )$. By the $\mbf{a}$-quasiconvexity of $F$ at $z$ we have for $\varepsilon > 0$
that
\begin{eqnarray*}
  0 &\leq& \int_{\Rn} \! \bigl( F(z+\varepsilon \n \phi ) - F(z) \bigr) \, \dd x\\
  &=& \int_{\Rn} \! \bigl( F(z+\varepsilon \n \phi ) - F(z) -F^{\prime}(z) \cdot \varepsilon \n \phi \bigr) \, \dd x\\
&=& \varepsilon^{2}\int_{\Rn} \! \int_{0}^{1} \! (1-t)F^{\prime \prime}(z+t\varepsilon \n \phi ) \bigl[ \n \phi , \n \phi \bigr] \, \dd t \, \dd x,
\end{eqnarray*}
and so upon division by $\varepsilon^2$ and taking the limit $\varepsilon \searrow 0$ we arrive at
$$
0 \leq \int_{\Rn} \! F^{\prime \prime}(z)\bigl[ \n \phi , \n \phi \bigr] \, \dd x .
$$
By a limiting argument we may extend the last inequality to all $\RN$-valued Schwartz test
maps $\phi$ on $\Rn$.
The real-valued symmetric bilinear form $F^{\prime \prime}(z)$ defined on $\M \times \M$ extends to $\Cp^{N \times m} \times \Cp^{N \times m}$ as a complex-valued
hermitian sesquilinear form by the definition
\begin{eqnarray}
  F^{\prime \prime}(z)\bigl[ w_{1},\overline{w_{2}} \bigr] &:=& F^{\prime \prime}(z)\bigl[ \mathrm{Re}(w_{1}),\mathrm{Re}(w_{2}) \bigr]
  + F^{\prime \prime}(z)\bigl[ \mathrm{Im}(w_{1}),\mathrm{Im}(w_{2}) \bigr]\label{complexified}\\
  &&+\mathrm{i}\bigl( F^{\prime \prime}(z)\bigl[ \mathrm{Im}(w_{1}), \mathrm{Re}(w_{2}) \bigr]-F^{\prime \prime}(z)\bigl[ \mathrm{Re}(w_{1}),\mathrm{Im}(w_{2}) \bigr]
  \bigr)\nonumber
\end{eqnarray}
for $w_1$, $w_{2} \in \Cp^{N \times m} \times \Cp^{N \times m}$. Using this definition and the previously derived bound for real-valued test maps
we find that
$$
0 \leq \int_{\Rn} \! F^{\prime \prime}(z)\bigl[ \n \phi , \overline{\n \phi} \bigr] \, \dd x 
$$
persists for Schwartz test maps $\phi \colon \Rn \to \Cp^{N}$.
Next we rewrite the right-hand side by use of the Fourier transform (defined component-wise for matrix valued $\LL^1$-maps $v$ on $\Rn$ by
$\widehat{v}( \xi ) = \int_{\Rn} \! v(x) \mathrm{e}^{-\mathrm{i}\xi \cdot x} \, \dd x$)
$$
\widehat{\n \phi}(\xi ) = \widehat{\phi}( \xi ) \otimes \sum_{\alpha \in \mathcal{H}} \bigl( \mathrm{i}\xi \bigr)^{\alpha} \mbf{e}_{\alpha} ,
$$
and so by Plancherel's theorem we get
$$
0 \leq \int_{\Rn} \! F^{\prime \prime}(z)\bigl[ \widehat{\phi}( \xi ) \otimes \sum_{\alpha \in \mathcal{H}} \bigl( \mathrm{i}\xi \bigr)^{\alpha} \mbf{e}_{\alpha} ,
\overline{\widehat{\phi}( \xi )} \otimes \sum_{\alpha \in \mathcal{H}} \bigl( -\mathrm{i}\xi \bigr)^{\alpha} \mbf{e}_{\alpha} \bigr] \, \dd \xi .
$$
Fix $v \in \RN$ and let $\eta \in \CC^{\infty}_{c}( \Rn )$ be real-valued. Take the Schwartz test map $\phi$ such that $\widehat{\phi} = v \eta$
and record
\begin{eqnarray*}
  0 &\leq& \int_{\Rn} \! \eta^{2} F^{\prime \prime}(z)\bigl[ v \otimes \sum_{\alpha \in \mathcal{H}} \bigl( \mathrm{i}\xi \bigr)^{\alpha} \mbf{e}_{\alpha} ,
    v \otimes \sum_{\alpha \in \mathcal{H}} \bigl( -\mathrm{i}\xi \bigr)^{\alpha} \mbf{e}_{\alpha} \bigr] \, \dd \xi\\
  &=& \int_{\Rn} \! \eta^{2} \sum_{\alpha , \beta \in \mathcal{H}} \mathrm{i}^{| \alpha | + | \beta |} (-1)^{| \beta |}
  F^{\prime \prime}(z)\bigl[ v \otimes \xi^{\alpha} \mbf{e}_{\alpha}, v \otimes \xi^{\beta} \mbf{e}_{\beta} \bigr] \, \dd \xi .
\end{eqnarray*}
Because $\eta$ is arbitrary here we must have the pointwise inequality
$$
0 \leq \sum_{\alpha , \beta \in \mathcal{H}} \mathrm{i}^{| \alpha | + | \beta |} (-1)^{| \beta |}
  F^{\prime \prime}(z)\bigl[ v \otimes \xi^{\alpha} \mbf{e}_{\alpha}, v \otimes \xi^{\beta} \mbf{e}_{\beta} \bigr]
$$
valid for all $\xi \in \Rn$.
Let us introduce the short-hand $c_{\alpha , \beta} = F^{\prime \prime}(z)\bigl[ v \otimes \xi^{\alpha} \mbf{e}_{\alpha}, v \otimes \xi^{\beta} \mbf{e}_{\beta} \bigr]$.
Using $\mathcal{H}_{e} = \mathcal{H}_{0} \cup \mathcal{H}_{2}$, $\mathcal{H}_{o}=\mathcal{H}_{1} \cup \mathcal{H}_{3}$ and noting
$c_{\alpha , \beta}=c_{\beta , \alpha}$, we find
$$
\sum_{\alpha \in \mathcal{H}_{e}, \beta \in \mathcal{H}_{o}} \mathrm{i}^{| \alpha | + | \beta |} (-1)^{| \beta |}c_{\alpha , \beta} =
-\sum_{\alpha \in \mathcal{H}_{o}, \beta \in \mathcal{H}_{e}} \mathrm{i}^{| \alpha | + | \beta |} (-1)^{| \beta |}c_{\alpha , \beta}.
$$
Therefore
$$
\sum_{\alpha , \beta \in \HH} \mathrm{i}^{| \alpha | + | \beta |} (-1)^{| \beta |}c_{\alpha , \beta} = \sum_{\alpha , \beta \in \HH_{0} \cup \HH_{2}} \mathrm{i}^{| \alpha | + | \beta |} c_{\alpha , \beta}
- \sum_{\alpha , \beta \in \HH_{1} \cup \HH_{3}} \mathrm{i}^{| \alpha | + | \beta |} c_{\alpha , \beta} .
$$
We split the first sum on the right-hand side as
$$
\sum_{\alpha , \beta \in \HH_{0} \cup \HH_{2}} \mathrm{i}^{| \alpha | + | \beta |} c_{\alpha , \beta} = \sum_{\alpha , \beta \in \HH_{0}} c_{\alpha , \beta} + \sum_{\alpha , \beta \in \HH_{2}} c_{\alpha , \beta}
-2\sum_{\alpha \in \HH_{0}, \beta \in \HH_{2}} c_{\alpha , \beta}
$$
and the second sum as
$$
- \sum_{\alpha , \beta \in \HH_{1} \cup \HH_{3}} \mathrm{i}^{| \alpha | + | \beta |} c_{\alpha , \beta} =\sum_{\alpha , \beta \in \HH_{1}} c_{\alpha , \beta} + \sum_{\alpha , \beta \in \HH_{3}} c_{\alpha , \beta}
-2\sum_{\alpha \in \HH_{1}, \beta \in \HH_{3}} c_{\alpha , \beta} .
$$
Adding these sums yields the required conclusion \eqref{arc}.
\end{proof}

\begin{lemma}\label{spanninglem}
Assume that the pattern of homogeneity satisfies $\HH = \HH_{j}$ for some $j \in \{ 0,1,2,3 \}$. Then we have
$$
\mathrm{span } \biggl\{ v \otimes \sum_{\alpha \in \HH} \xi^{\alpha} \mbf{e}_{\alpha} : \, v \in \RN , \, \xi \in \Rn \biggr\} = \M .
$$
\end{lemma}
The proof is identical to the one given in \cite[Lemma 12]{kazstowoj17} and so is omitted here. Note that the results in \cite{kazstowoj17}
that depend on $\mbf{a}$-rank-one convexity being convexity in the directions of a spanning cone therefore remain valid when $\HH = \HH_j$ for
some $j$. The results in \cite{Pros,Pros2} do not depend on this as long as one assumes the considered integrands satisfy a suitable
Lipschitz condition.

Let us finally record that the full Van Hove's theorem \cite{vanhove} remains valid in the context of mixed smoothness.
A quadratic form $\mathcal{Q}$ on $\M$ is called $\mbf{a}$-rank-one convex provided \eqref{arc} holds.

\begin{corollary}\label{vanhove-theorem}
Let $\mathcal{Q} \colon \M \to \R$ be a quadratic form. Then $\mathcal{Q}$ is $\mbf{a}$-rank-one convex if and only if it is $\mbf{a}$-quasiconvex.
\end{corollary}
\end{section}

\begin{section}{Linear theory}
Throughout this section we fix $Q=Q_{1}(0)$ and state all results on $Q$. The corresponding results on a general anisotropic cube $Q_{R}(x_{0})$
are then obtained by translation and scaling. 
\begin{definition}
We say that a constant, symmetric tensor $T$ is strictly $\mbf{a}$-quasielliptic with ellipticity constant $\nu >0$ if 
\begin{equation}\label{quasiellipiticity eqn}
\int_{Q} \! T\n \phi \cdot \n \phi \, \dd x \geq \nu \int_{Q} \! |\n \phi |^2 \, \dd x
\end{equation} for all $\phi \in \WW_0^{\mbf{a},2}(Q)$. 
\end{definition}
In this section we study weak solutions of
\begin{equation}\begin{cases}
    \n^* T \n u=f \text{ in } Q,\\
    u \in \WW_0^{\mbf{a},p}(Q)\end{cases}
\end{equation}
where $T$ is $\mbf{a}$-quasielliptic. Here $\n^*$ denotes the distributional adjoint of $\n$.
In general we assume $f \in \WW^{-\mbf{a},p}(Q) := (\WW_0^{\mbf{a},p'}(Q))^*$. One may show, as in the standard case, that any $f \in \WW^{-\mbf{a},p}(Q)$ may be
written as $f= \sum_{\beta \in \mathcal{A}} \p^\beta f_\beta$ for some $f_\beta \in \LL^p(Q)$, and
$\Vert f\Vert_{\WW^{-\mbf{a},p}(Q)}\approx \inf_{\left\{f_\beta\right\}_\beta} \sum_{\beta \in \mathcal A} \Vert f_\beta\Vert_{\LL^p(Q)}$, where the infimum is taken over all choices
of $\left\{f_\beta \right\}_\beta$ for which the representation $f= \sum_{\beta \in \mathcal{A}} \p^\beta f_\beta$ holds. We begin by observing that, analogously to
the isotropic case, $\mbf{a}$-quasiellipticity (\ref{quasiellipiticity eqn}) is equivalent to a Legendre--Hadamard-type condition.

\begin{lemma}\label{lemmaLegendreHadamard}
A constant, symmetric tensor $T$ parametrised by pairs $( (\alpha, j), (\widetilde{\alpha}, \widetilde{j}))$ with
$\langle \alpha, \mbf{a}^{-1} \rangle = \langle \widetilde{\alpha}, \mbf{a}^{-1} \rangle = 1$ and $j, \widetilde{j} \in \{1, \ldots, N\}$ is strictly
quasielliptic with quasiellipticity constant $\nu$ if and only if, for all $\xi \in \Rn$ and all $b \in \RN$,
\begin{equation}\label{eqStrictQuasiconvexity}
\begin{aligned}
\nu \left( \sum_{\langle \alpha, \mbf a^{-1} \rangle = 1} 
\xi^{2\alpha} \right) |b|^2 \leq & \sum_{\alpha, \widetilde{\alpha} \in \mathcal{H}_e} \sum_{1 \leq j, \widetilde{j} \leq N} T_{(\widetilde{\alpha}, \widetilde{j})}^{(\alpha, j)}
(-1)^{(|\alpha| + |\widetilde{\alpha}|)/2} \xi^{\alpha + \widetilde{\alpha}} b_j b_{\widetilde{j}} + \\
&  \sum_{\beta, \widetilde{\beta} \in \mathcal{H}_o} \sum_{1 \leq j,\widetilde{j} \leq N} T_{(\widetilde{\beta}, \widetilde{j})}^{(\beta, j)} (-1)^{(|\beta| + |\widetilde{\beta}| + 2)/2}
\xi^{\beta + \widetilde{\beta}} b_j b_{\widetilde{j}},
\end{aligned}
\end{equation}
where we, as above, have let $\mathcal{H}_e$ and $\mathcal{H}_o$ be the subsets of $\{ \alpha \colon \langle \alpha, \mbf a^{-1} \rangle = 1\}$ of even and
odd degrees respectively.
\end{lemma}
\begin{proof}
The `only if' part is an immediate corollary of Proposition \ref{propLegendreHadamard}.
Conversely, let $T$ be a tensor satisfying our assumptions, and let us fix an arbitrary function $\phi \in \WW^{\mbf a,2}_0(Q)$.
We may extend $\phi$ by $0$ to the whole of $\Rn$ and, using Plancherel's theorem, write 
\begin{equation*}
\begin{aligned} 
\int_Q \! \langle T \n \phi, \n \phi \rangle \, \dd x &=  \int_{\Rn} \! \langle \widehat{T \n \phi}, \widehat{\n \phi} \rangle \dd \xi \\ 
&= \int_{\Rn} \! \sum_{\substack{(\alpha, j) \\ (\widetilde{\alpha}, \widetilde{j})}} T^{\alpha, j}_{\widetilde{\alpha},\widetilde{j}} (i \xi)^\alpha \widehat{\phi}^j
\overline{(i\xi)^{\widetilde{\alpha}} \widehat{\phi}^{\widetilde{j}}} \, \dd \xi, 
\end{aligned}
\end{equation*}
where $\h{f}$ denotes the Fourier transform of $f$ and $\overline{z}$ is the complex conjugate of $z$.
Since $T$ is symmetric the terms where the orders of $\alpha$ and $\widetilde{\alpha}$ have different parities cancel out due to the $i^{|\alpha|}(-i)^{|\widetilde{\alpha}|}$ term.
Thus, we are left with exactly the expression on the right-hand side of \eqref{eqStrictQuasiconvexity} with $b = \h{\varphi}$.
Applying the inequality \eqref{eqStrictQuasiconvexity} separately to real and imaginary parts of $\h{\phi}$ we may estimate with
$$
\nu \int_{\Rn} \! \left( \sum_{\langle \alpha, \mbf a^{-1} \rangle = 1} \xi^{2\alpha} \right) \h{\phi} \overline{\h{\phi}} \, \dd \xi =
\nu \int_{\Rn} \!  \sum_{\langle \alpha, \mbf a^{-1} \rangle = 1}  \h{\p^\alpha \phi} \overline{\h{\p^{\alpha} \phi}} \, \dd \xi,
$$
and using Plancherel's identity again ends the proof.
\end{proof}
Next we recall the notion of hypoellipticity and semiellipticity from \cite{hormander1969linear}.
We say that a differential operator $P(D)$ is hypoelliptic if for any open set $A \subset \Rn$ and a distribution $u \in \mathcal{D}'(A)$ the equation
$P(D) u = 0$ in $\mathcal{D}'(A)$ implies $u \in \CC^{\infty}(A)$. We say $P(D)$ is semielliptic if there exists a vector
$\gamma = (\gamma_1, \ldots, \gamma_n) \in \N^n$  such that
$$
P(D) = \sum_{\langle \alpha, \gamma^{-1} \rangle \leq 1} a_\alpha \p^{\alpha}
$$
with some coefficients $a_\alpha \in \R^{N \times N}$, and if we set 
$$
P^0(D) := \sum_{\langle \alpha, \gamma^{-1}\rangle = 1} a_\alpha  \p^{\alpha}
$$
then $\det(P^0(-i \xi)) \not= 0$ for all $\xi \in \Rn \setminus \{0\}$.
It is well known that semielliptic operators are hypoelliptic (see \cite[Theorem 4.1.8]{hormander1969linear}).

\begin{corollary}\label{propAEllipticIsHypoelliptic}
If $T$ is a constant, strictly $\mbf{a}$-quasielliptic tensor then the operator $\n^{*} T \n$ is semielliptic and therefore in particular hypoelliptic.
\end{corollary}

\begin{proof}
Note the operator $\n^{*} T \n$ only involves derivatives $\p^{\alpha + \beta}$ on the hyperplane of homogeneity given through
$\langle \alpha + \beta, \mbf{a}^{-1} \rangle = 2$, so the only thing that we need to check is that its determinant is non-zero.
To do so we write 
$$
\n^{*} T \n (-i\xi) = \sum_{\substack{\langle \alpha, \mbf a^{-1} \rangle = 1 \\
\langle \widetilde{\alpha}, \mbf a^{-1} \rangle = 1}} T^{\alpha}_{\widetilde{\alpha}} (-i \xi)^{\alpha} (-1)^{|\widetilde{\alpha}|}(-i \xi)^{\widetilde{\alpha}}.
$$
As noted before, when $\alpha$ and $\widetilde{\alpha}$ have orders of different parities the $(\alpha, \widetilde{\alpha})$ term cancels the
$(\widetilde{\alpha}, \alpha)$ one, due to the $(-1)^{|\widetilde{\alpha}|}$ factor coming from the adjoint $\n^{*}$. Hence, with $\mathcal{H}_e$
and $\mathcal{H}_o$ as in the lemma above,
\begin{equation*}
  \n^{*} T \n (-i\xi) = \sum_{\alpha, \widetilde{\alpha} \in \mathcal{H}_e} T^{\alpha}_{\widetilde{\alpha}} (-1)^{(|\alpha| + |\widetilde{\alpha}|)/2} \xi^{\alpha + \widetilde{\alpha}} +
  \sum_{\alpha, \widetilde{\alpha} \in \mathcal{H}_o} T^{\alpha}_{\widetilde{\alpha}} (-1)^{(|\alpha| + |\widetilde{\alpha}| + 2)/2} \xi^{\alpha + \widetilde{\alpha}}.
\end{equation*}
Therefore, the inequality \eqref{eqStrictQuasiconvexity} may be rewritten as
$$
b^T \left(\n^{*} T \n (-i\xi) \right) b \geq \nu \left( \sum_{\langle \alpha, \mbf a^{-1} \rangle = 1} \xi^{2 \alpha} \right) \|b\|_2^2 > 0
$$
for every $0 \not= b \in \RN$. This means that $\n^{*} T \n (-i\xi) $ is a positive definite matrix, from which we infer that its determinant is
strictly positive, thus ending the proof.
\end{proof}
We recall two lemmas essentially due to Giusti (see \cite{GiustiItalian}).

\begin{lemma}\cite[Corollary 8.11]{GiustiItalian}\label{lemmaZeroRHSEstimate}
Let $T$ be a strictly $\mbf a$-quasielliptic tensor. Let $u \in C^{\infty}({Q_r})$ solve $\n^{*} T \n u = 0$ in $Q_r \Subset Q.$
Then, for any $\alpha$ with $\langle \alpha, \mbf a^{-1} \rangle \leq 1$, any radius $\rho < r$, any $p \geq 2$, and any constant $b$ we have
$$
\| \p^\alpha u - (\p^\alpha u)_\rho \|^p_{\LL^{p}(Q_\rho)} \leq C \left(\frac{\rho}{r}\right)^{|\mbf a^{-1}| + (p/\max_i a_i)} \| \p^\alpha u - b \|^p_{\LL^{p}(Q_r)}
$$
and
$$
\Vert \p^\alpha u \Vert_{\LL^{p}(Q_\rho)}^p \leq C\left( \frac{\rho}{r}\right)^{|\mbf{a}^{-1}|} \Vert \p^\alpha u\Vert_{\LL^{p}(Q_r)}^p,
$$
where the constant $C = C(\nu)$ only depends on the $\mbf a$-quasiellipticity constant of $T$.
\end{lemma}
In fact, it is possible to show the above decay estimates hold for all $1\leq p<\infty$, but we do not pursue this here.

\begin{lemma}\cite[Theorem 12.11]{GiustiItalian}\label{thm12IIGiusti}
Suppose that a function $u$ solves
\begin{equation*}
\begin{cases}
\n^{*} T \n u = \sum_{\langle \alpha, \mbf a^{-1} \rangle \leq 1} \p^{\alpha} f_{\alpha} \text{ in } Q,\\
u \in \WW^{\mbf{a},2}_{0}(Q),
\end{cases}
\end{equation*}
where $T$ is a constant strictly $\mbf{a}$-quasielliptic tensor. Assume that all the functions $f_\alpha$ for $\alpha \not= 0$ are of class
$\mathfrak{L}_{\mbf{a}}^{2,1}(Q)$, and that $f_0 \in \LL^\infty(Q)$. Then, for all $\alpha$ with $\langle \alpha, \mbf{a}^{-1} \rangle \leq 1$, we have
$\p^\alpha u \in \mathfrak{L}_{\mbf{a}}^{2,1}(Q)$ and
$$
\sum_{\langle \alpha, \mbf a^{-1} \rangle \leq 1} \| \p^\alpha u \|_{\mathfrak{L}_{\mbf{a}}^{2,1}(Q)}^2 \leq C\left( \|u\|_{\WW^{\mbf{a},2}(Q)}^2
+ \sum_{0 < \langle \alpha, \mbf a^{-1} \rangle \leq 1} \| f_\alpha \|_{\mathfrak{L}_{\mbf{a}}^{2,1}(Q)}^2 + \|f_0\|_{L^\infty(Q)}^2 \right),
$$
where the constant $C = C(\nu)$ only depends on the $\mbf{a}$-quasiellipticity constant of $T$.
\end{lemma}
The above lemma differs slightly from Giusti's paper \cite{GiustiItalian} in we explicitly include lower order terms on the right-hand side instead
of $\sum_{\langle \alpha, \mbf{a}^{-1} \rangle = 1} \p^{\alpha} f_{\alpha}$. Despite this, the proof remains basically unchanged,
and full details can be found in \cite{Pros}.  
\begin{theorem}
\label{propLpBoundsLinear}
Let $T$ be a constant, strictly $\mbf a$-quasielliptic tensor. For any $2 \leq p < \infty$ the equation
\begin{equation}\label{eqPropLpBounds}
\begin{cases}
\n^{*} T \n u = f \text{ in } Q,\\
u \in \WW^{\mbf{a},p}_{0}(Q)
\end{cases}
\end{equation}
with $f \in \WW^{-\mbf{a},p}(Q)$ is uniquely solvable and the solution $u$ satisfies 
$$
\|u\|_{\WW^{\mbf{a},p}_{0}(Q)} \leq C \|f\|_{\WW^{-\mbf{a},p}(Q)},
$$ 
with the constant $C = C(p, \nu)$ depending only on $p$ and the ellipticity constant of $T$.
\end{theorem}

\begin{proof}[Proof of Theorem \ref{propLpBoundsLinear}]
Note that for $p=2$ the theorem is a consequence of the Lax-Milgram theorem. Thus, we only need to justify the $\LL^p$ bounds for $p > 2$,
which we justify via interpolation (see \cite[Chapter 10]{GG}). We need to show that the operators mapping $(f_\alpha)_{\alpha \not= 0}$ to the
$\mbf a$-gradient of the solution of 
\begin{equation}\label{eqLpBoundsHigherOrder}
\begin{cases}
\n^{*} T \n u = \sum_{0 < \langle \alpha, \mbf a^{-1} \rangle \leq 1} \p^\alpha f_\alpha \text{ in } Q,\\
u \in \WW^{\mbf{a},p}_{0}(Q),
\end{cases}
\end{equation}
and mapping $f_0$ to the $\mbf{a}$-gradient of the solution of
\begin{equation}\label{eqLpBoundsZeroOrder}
\begin{cases}
\n^{*} T \n u = f_0 \text{ in } Q,\\
u \in \WW^{\mbf{a},p}_{0}(Q),
\end{cases}
\end{equation}
are of strong $(p,p)$ type with norms bounded by a constant depending only on $p$ and the strict $\mbf a$-quasiellipticity constant of $T$. 

We know from the Lax--Milgram theorem that both operators are of strong type $(2,2)$. On the other hand, Lemma \ref{thm12IIGiusti} tells
us that the operator mapping $(f_\alpha)_\alpha$ to the $\mbf a$-gradient of the solution of \eqref{eqLpBoundsHigherOrder} maps $\mathfrak{L}_{\mbf{a}}^{2,1}(Q)$
to $\mathfrak{L}_{\mbf{a}}^{2,1}(Q)$ continuously. The same estimate shows that the operator given by \eqref{eqLpBoundsZeroOrder} is continuous as a mapping
from $\LL^\infty(Q)$ to $\mathfrak{L}_{\mbf{a}}^{2,1}(Q)$. Hence, the Stampacchia interpolation theorem (see \cite[Theorem 2.14]{GG} and 
\cite[Theorem 2.1]{GiustiItalian}) lets us conclude that both operators are of strong type $(p,p)$, and thus so is their sum, which ends the proof.
\end{proof}
\end{section}

\begin{section}{Auxiliary lemmas}
In this section we collect some elementary lemmas that will be needed to carry out the harmonic approximation procedure in Section 5. 
Throughout we write $V_p(z)=(1+|z|^2)^{p/2}-1$, where $1<p<\infty$. We recall some standard facts about $V_p$ \cite[Lemmas 3.6-3.7]{gk24}.

\begin{lemma}\label{lem: new V function lemma}
Let $1<p<\infty$. Then
\begin{enumerate}[(i)]
    \item $V_p(z) \approx |z|^2 \mbf{1}_{|z| \leq 1}+|z|^p \mbf{1}_{|z| > 1}$. In particular, if $p \geq 2$ then $V_p(z) \approx |z|^2+|z|^p$
    \item $t \mapsto V_p(t^{1/r})$ is convex for all $1 \leq r \leq \min (2,p)$
    \item $V_p(\lambda z) \leq \begin{cases} 
\lambda ^{\max\left\{2,p\right\}}V_p(z) & \text{if } \lambda \geq 1\\
\lambda ^{\min\left\{2,p\right\}}V_p(z) & \text{if } \lambda \leq 1
\end{cases}$
    \item $V_p(z+w) \lesssim V_p(z)+V_p(w)$
    \item $V_p(z+w)-V_p(z)-V_p'(z)\cdot w \approx (1+|z|^2+|w|^2)^{\frac{p-2}{2}}|w|^2$.
\end{enumerate}
\end{lemma}
Given $F \in \CC^{2}(\M )$, for each $z_0 \in \M$, the shifted integrand is defined as
$$
F_{z_0}(z)=F(z+z_0)-F(z_0)-F'(z_0) \cdot z.
$$ 

\begin{lemma}\label{lem:equivalence of quasiconvxity}
Let $1<p<\infty$, $F \in \CC^{2}(\M )$ satisfy (H1). Then (H3) is equivalent to the statement: There exists $\lambda_M>0$ such that
$\lambda_M \int_{Q_R}V_p(\n \varphi) \, \dd x \leq \int_{Q_R} F_{z_0}(\n \varphi) \, \dd x $ for all $|z_0| \leq M$ and $\varphi \in \WW_{0}^{\mbf{a},p}(Q_R , \RN )$.

\end{lemma}
\begin{proof}
Let $|z_0| \leq M$ and $\varphi \in \WW_0^{\mbf{a},p}(Q_R ,\RN )$. From Lemma \ref{lem: new V function lemma} (v) we have that
$$
V_p(\n \phi+z_{0})-V_{p}(z_{0})-V_{p}'(z_{0}) \cdot \n \phi \approx (1+|\n \phi |^{2}+|z_{0}|^{2})^{\frac{p-2}{2}}|\n \phi |^{2}
\approx_M (1+|\n \phi |^{2})^{\frac{p-2}{2}}|\n \phi|^{2} \approx V_{p}(\n \phi ).
$$
From the divergence theorem it therefore follows
$$
\int_{Q_R} \! V_{p}(\n \phi +z_{0})-V_{p}(z_{0}) \, \dd x \approx_M \int_{Q_{R}}V_{p}(\n \phi) \, \dd x.
$$
The claim then readily follows from the definition of $F_{z_0}$ and the divergence theorem. 
\end{proof}

\begin{lemma}\label{lem: linearisation of F}
Let $1<p<\infty$, $F \in \CC^{2}(\M )$ satisfy (H1) and (H2). Then for every $|z_0| \leq M$, the following hold:

\begin{enumerate}[(i)]
    \item $|F_{z_0}(z)| \lesssim_{M} V_p(z)$
    \item $|F'_{z_0}(z)| \lesssim_{M} V_p(z)/|z|$
    \item $|F''(z_0)z-F'_{z_0}(z)| \lesssim \omega_M(|z|)(|z|+|z|^{p-1})$ where $\omega_M \leq 1 $ is a non-negative concave increasing modulus of continuity.
\end{enumerate}
\end{lemma}

\begin{proof}
If $|z| \geq 1$ all estimates follow directly from the growth assumptions. For $|z| \leq 1$ we estimate
$|F_{z_0}(z)| \leq \int_0^1 \int_0^1| F''(z_0+stz)z\cdot z t |\dd s \dd t \leq \max_{|w| \leq M+1}|F''(w)| \cdot |z|^2 \lesssim_M V_p(z)$ which proves (i),
and similarly for (ii). For (iii), note that there exists a modulus of continuity $\omega_M \leq 1$ such that $|F''(v)-F''(w)| \leq C_M \omega_M(|v-w|)$
for all $|v|, |w| \leq {M+1}$. Therefore $|F''(z_0)z-F'_{z_0}(z)| \leq \int_0^1 |F''(z_0+tz)-F''(z_0) | \dd t |z| \leq C_M \omega_M (|z|)|z|$. 
\end{proof}
The following lemma is a simple variant of the standard iteration lemma  \cite[Lemma 6.1]{GG} for which we sketch the details.

\begin{lemma}\label{lem: iteration for V function}
Let $1<p<\infty$, and let $f \colon [R/2, R] \to [0,\infty)$ be a bounded function satisfying for all $R/2 \leq t < s \leq R,$
$$
f(t) \leq \theta f(s) + \sum_{0 \neq \beta \leq \alpha, \alpha \in \mathcal{H}} \int_{Q_R}V_p\left(\frac{\p^{\alpha-\beta}v}{\prod_{i=1}^n(s^{1/a_i}-t^{1/a_i})^{\beta_i}} \right) \, \dd x
$$
where $\theta \in [0,1)$ and $v \in \WW^{\mbf{a},p}(Q_R)$. Then there exists $C=C_{\mbf{a}, \theta, p}>0$ such that
$$
f(R/2) \leq C  \sum_{0 \neq \beta \leq \alpha, \alpha \in \mathcal{H}} \int_{Q_R}V_p\left(\frac{\p^{\alpha-\beta}v}{R^{\langle \beta, \mbf{a}^{-1}\rangle }}\right) \, \dd x
$$ 
\end{lemma}
\begin{proof}
By applying the mean value theorem to $x \mapsto x^{1/a_i}$, we see that
$$
\frac{1}{a_i}R^{1/a_i-1}(s-t) \leq s^{1/a_i}-t^{1/a_i} \leq \frac{1}{a_i}(R/2)^{1/a_i-1}(s-t)
$$
for all $R/2 \leq t<s\leq R$. In particular
$$
f(t) \leq \theta f(s)+C\sum_{0 \neq \beta \leq \alpha, \alpha \in \mathcal{H}} \int_{Q_R}V_p\left(\frac{\p^{\alpha-\beta}v}{R^{\langle \beta, \mbf{a}^{-1} \rangle -|\beta|}(s-t)^{|\beta|}} \right) \dd x.
$$
Now in the standard fashion, select $r_j=R/2 + (1-\lambda^j)(R/2)$ where $\lambda \in (0,1)$ will be specified.
Then it follows $V_p\left(\frac{\p^{\alpha-\beta}v}{R^{\langle \beta, \mbf{a}^{-1} \rangle -|\beta|}(r_{j+1}-r_j)^{|\beta|}} \right) \leq
C\lambda^{-|\beta|j  \max(2,p)}V_p\left(\frac{\p^{\alpha-\beta}v}{R^{\langle \beta, \mbf{a}^{-1}\rangle }} \right)$. Thus, iteration and boundedness of $f$ guarantees that 
$$
f(R/2) \leq C\sum_{j=0}^\infty \left(\frac{\theta}{\lambda^{{\max_{\beta \in \mathcal{A}} |\beta|} \max(2,p)
}}\right)^j \sum_{0 \neq \beta \leq \alpha, \alpha \in \mathcal{H}} \int_{Q_R}V_p\left(\frac{\p^{\alpha-\beta}v}{R^{\langle \beta, \mbf{a}^{-1}\rangle }}\right) \, \dd x.
$$
Choosing $\lambda$ sufficiently close to 1 so that $\frac{\theta}{\lambda^{{\max_{\beta \in \mathcal{A}}} |\beta| \max(2,p)}} <1$ we are done.  
\end{proof}
\end{section}

\begin{section}{A-harmonic approximation and partial regularity}
Let $1<p<\infty$ and $\Omega \subset \Rn$ an open set. We assume that $F \in \CC^{2}(\M) $ satisfies

\begin{enumerate}[(H1)]
    \item $|F(z)| \leq L(1+|z|^p)$
    \item $|F'(z)| \leq L(1+|z|^{p-1})$
    \item For every $M>0$ there exists $l_M >0$ such that $F-l_MV_p$ is $\mbf{a}$-quasiconvex at $z_0$ for all $|z_0| \leq M$. 
\end{enumerate}
We say $u \in \WW^{\mbf{a},p}_{\text{loc}}(\Omega ,\RN )$ is a local minimiser of
$\mathcal{F}[v;\Omega] = \int_{\Omega}F(\n v) \, \dd x$ if $\mathcal{F}[u;\Omega'] \leq \mathcal{F}[u+\varphi; \Omega']$ whenever
$\varphi \in \WW_0^{\mbf{a},p}(\Omega',\RN)$ and $\Omega' \Subset \Omega$. 

\begin{prop}\label{prp caccioppoli}
Let $1<p<\infty$ and $u \in \WW^{\mbf{a},p}_{\text{loc}}(\Omega ,\RN )$ be a local minimiser of $\mathcal{F}$. Then 
\begin{equation}\label{eqnCaccioppoli}\strokedint_{Q_R} V_p(\n u-z_0) \, \dd x \lesssim_M
  \strokedint_{Q_{2R}} \sum_{0 \neq \beta \leq \alpha, \alpha \in \mathcal{H}} V_p\left( \frac{\p^{\alpha-\beta}(u-b)}{R^{\langle \beta, \mbf{a}^{-1} \rangle}}\right) \, \dd x
\end{equation}
for all $Q_{2R} \Subset \Omega$, all $|z_0| \leq M$ and for all polynomials $b$ satisfying $\n b=z_0$. 
\end{prop}
\begin{proof}
Fix $R \leq t <s \leq 2R$. By considering products of one dimensional cut-off functions, we can construct $\eta \in C_c^\infty(Q_s;[0,1])$ such that
$\eta=1$ on $Q_t$ and $|\p^\beta \eta| \lesssim \prod_{i=1}^n(s^{1/a_i}-t^{1/a_i})^{-\beta_i}$ for all $\beta$. We now let $\varphi = \eta (u-b) \in W_0^{\mbf{a},p}(Q_s)$
and $v=u-\varphi$. Observe that
$$
\n v-z_0=(1-\eta)(\n u-z_0)-\left(\sum_{0 \neq \beta \leq \alpha} {\alpha \choose \beta }\p^\beta \eta \p^{\alpha-\beta}(u-b) \right)_{\alpha \in \mathcal{H}}.
$$
Now from (H3) (see Lemma \ref{lem:equivalence of quasiconvxity}) we have that
\begin{align*}\int_{Q_s} V_p(\n \varphi) \dd x &\lesssim_M \int_{Q_s}F_{z_0}(\n \varphi) \, \dd x\\
&= \int_{Q_s}F_{z_0}(\n u-z_0) \dd x + \int_{Q_s \setminus Q_t} F_{z_0}(\n \varphi)-F_{z_0}(\n u-z_0) \, \dd x =: (I)+(II)
\end{align*}
From the minimisation property and Lemma \ref{lem: linearisation of F} we have
\begin{align*}
(I) & \leq  \int_{Q_s}F_{z_0}(\n v-z_0) \, \dd x \\
& \lesssim \int_{Q_s \setminus Q_t}V_p(\n v-z_0) \, \dd x\\
&\lesssim \int_{Q_s\setminus Q_t}V_p(\n u-z_0) \, \dd x +
\int_{Q_{2R}} V_p \left( \left( \sum_{0 \neq \beta \leq \alpha} \frac{\p^{\alpha-\beta}(u-b)}{\prod_{i=1}^n (s^{1/a_i}-t^{1/a_i})^{\beta_i}}\right)_{\alpha \in \mathcal{H}}\right) \, \dd x.
\end{align*}
We simply estimate $(II)$ as
\begin{align*}
(II) \leq \int_{Q_s\setminus Q_t}V_p(\n \varphi) + V_p(\n u-z_0) \dd x \lesssim \int_{Q_s \setminus Q_t} V_p(\n u-z_0) + V_p(\n v-z_0) \, \dd x.
\end{align*}
Applying the hole filling method, we arrive at
$$
\int_{Q_t}V_p(\n u-z_0) \, \dd x \leq
\theta \int_{Q_s}V_p(\n u-z_0) \dd x + C \sum_{0 \neq \beta \leq \alpha, \alpha \in \mathcal{H}}
\int_{Q_{2R}} \! V_{p}\left(  \frac{\p^{\alpha-\beta}(u-b)}{\prod_{i=1}^n (s^{1/a_i}-t^{1/a_i})^{\beta_i}}\right) \, \dd x
$$
for some $\theta \in [0,1)$. The claim now follows from Lemma \ref{lem: iteration for V function}. 
\end{proof}
\begin{lemma}\label{lem: quasiellipticity of F}
Let $1<p<\infty$. For every $|z_0| \leq M$, $F''(z_0)$ is strictly $\mbf{a}$-quasielliptic (with constant depending on $M$). 
\end{lemma}
\begin{proof}
Given $\varphi \in \CC_c^\infty(Q)$, consider the mapping $\Theta(t)= \int_{Q} F(z_0+t \n \varphi)-F(z_0)-\lambda_M V_p(t\n \varphi) \dd x  \geq 0.$
Then $\Theta$ has a minimum at $t=0$ so $\Theta''(0) \geq 0$.
A straightforward computation yields $\Theta''(0) = \int_{Q} F''(z_0) \n \varphi \cdot \n \varphi - p \lambda_M |\n \varphi|^2 \dd x$, from which the claim now follows by density. 
\end{proof}
We are now ready to carry out the harmonic approximation procedure.
For technical reasons we will restrict to the case $p\geq 2$; the subquadratic case $1<p<2$ will be addressed in forthcoming work.

\begin{prop}\label{prop A harmonic approx}
Let $2 \leq p<\infty$ and $u \in \WW^{\mbf{a},p}_{\text{loc}}(\Omega ,\RN )$ be a solution of the Euler--Lagrange equation $- \n^*F'(\n u)=0$.
Fix $M>0$, and let $a$ be any polynomial such that $\n a =z_0$ where $|z_0| \leq M$. Set $\tilde u=u-a$.
Let $\tilde h \in \tilde u+W_0^{\mbf{a},p}(Q_R;\R^N)$ be the unique solution of the linear equation $-\n^* F''(z_0) \n \tilde h=0$ in $Q_R$. Then
\begin{equation}
  \sum_{\langle \beta, \mbf{a}^{-1}\rangle<1} \strokedint_{Q_R} V_p\left(\frac{\p^\beta(\tilde u-\tilde h)}{R^{1-\langle \beta, \mbf{a}^{-1}\rangle}} \right) \, \dd x \leq
  C_M H\left(\strokedint_{Q_R} V_p(\n \tilde u) \dd x \right)
\end{equation}
where $H \colon [0,\infty) \to [0,\infty)$ is a continuous function depending only on $p, F,M$ such that
$$
\lim_{t \searrow 0} \frac{H(t)}{t}=0.
$$
\end{prop} 

\begin{proof}
By standard scaling arguments, it suffices to prove the case $R=1$. Note existence of $\tilde h$ follows from
Theorem \ref{propLpBoundsLinear} and Lemma \ref{lem: quasiellipticity of F}. Let us now fix $\langle \beta, \mbf{a}^{-1} \rangle <1$. 
Let $p_{\beta}>p$ be chosen so that $\langle \beta, \mbf{a}^{-1}\rangle+(1/p -1/p_\beta)|\mbf{a}^{-1}| \leq 1$, so that $p_\beta$ is a higher integrability
exponent from Proposition \ref{prp:sobolev embedding}. Observe that 
\begin{align*}
    \Vert (\p^\beta)^* \p^\beta(\tilde u-\tilde h)\Vert_{\WW^{-\mbf{a},p_\beta}(Q)} &= \sup_{\Vert \xi\Vert_{W_0^{\mbf{a},p_\beta'}} \leq 1} \int_{Q}\p^\beta(\tilde u-\tilde h)\cdot \p^\beta \xi \dd x\\
    & \leq \Vert\p^\beta(\tilde u-\tilde h) \Vert_{\LL^p} \sup_{\Vert \xi\Vert_{W_0^{\mbf{a},p_\beta'}} \leq 1} \Vert \p^\beta \xi \Vert_{\LL^{p'}}\\
    & \leq C \Vert \p^\beta (\tilde u-\tilde h) \Vert_{\LL^{p}(Q)}.
\end{align*}
Consequently there exists a unique $\varphi^\beta \in \WW_0^{\mbf{a},p_\beta}(Q ,\RN )$ solving
\begin{equation}\label{eqphibetapde}
\n^*F''(z_0) \n \varphi^\beta = (\p^\beta)^* \p^\beta (\tilde u-\tilde h)
\end{equation}
and satisfying the estimate $\Vert \n \varphi^\beta \Vert_{\LL^{p_\beta}(Q)} \leq C \Vert \p^\beta(\tilde u-\tilde h) \Vert_{\LL^{p}(Q)}$. 
Since $p \geq 2$, we can write
$$
\int_{Q}V_p(\p^\beta (\tilde u-\tilde h)) \dd x \approx \int_{Q} |\p^\beta (\tilde u-\tilde h)|^2+|\p^\beta (\tilde u-\tilde h)|^p \dd x =: (I)+(II).
$$
We begin by estimating $(II)$ (for which we may implicitly assume $p>2$). Since $\tilde u-\tilde h \in \WW_0^{\mbf{a},p}(Q)$ satisfies
$-\n^* F''(z_0) \n (\tilde u-\tilde h)=-\n^* F''(z_0) \n \tilde u$, we have that $\Vert \n (\tilde u-\tilde h)\Vert_{L^q} \leq C \Vert \n \tilde u\Vert_{L^q}$ for any $2 \leq q \leq p$.
In particular, we may choose $2 < q <p$ so that $\langle \beta, \mbf{a}^{-1}\rangle+(1/q - 1/p)|\mbf{a}^{-1}| \leq 1$,
so that by Sobolev embedding it follows $\Vert \p^\beta (\tilde u-\tilde h)\Vert_{\LL^p} \leq C\Vert \n( \tilde u -\tilde h)\Vert_{\LL^q}$. Combining these estimates we see that 
$$
(II) \leq C\left(\int_{Q}|\n \tilde u|^q \dd x\right)^{p/q} \leq C  \left(\int_{Q}|\n \tilde u|^2 +|\n \tilde u|^p \dd x\right)^{p/q} =C \left(\int_{Q}V_p(\n \tilde u) \dd x \right)^{p/q}.
$$
To estimate $(I)$, we use $\tilde u-\tilde h$ as a test function in (\ref{eqphibetapde}), and $\varphi^\beta$ as a test function for  $-\n^* F''(z_0) \n \tilde h=0$
and the Euler--Lagrange equation, to deduce

\begin{align*}
(I) & = \int_{Q}|\p^\beta (\tilde u-\tilde h)|^2 \dd x \\
& =\int_{Q} F''(z_0) \n \varphi^\beta  \cdot \n (\tilde u -\tilde h) \dd x\\
& = \int_{Q}F''(z_0) \n \tilde u \cdot \n \varphi^\beta \dd x\\
&=\int_{Q}(F''(z_0)\n \tilde u-F'_{z_0}(\n \tilde u)) \cdot \n \varphi ^\beta \dd x\\
& \leq \int_{Q}\omega_M(|\n \tilde u|)(|\n \tilde u|+|\n \tilde u|^{p-1})|\n \varphi^\beta | \dd x:= (III)+(IV).
\end{align*}
Since $p_\beta >p$, there exists $r_p>1$ such that $\frac{1}{p'}+\frac{1}{p_\beta}+\frac{1}{r_p}=1$. Then by Hölder's and Young's inequality
\begin{align*}
(IV) &\leq \Vert \omega_M(|\n \tilde u|)\Vert_{L^{r_p}} \Vert |\n \tilde u|^{p-1}\Vert_{L^{p'}}\Vert \n \varphi^\beta\Vert_{L^{p_\beta}}\\
& \leq C\Vert \omega_M(|\n \tilde u|)\Vert_{L^{r_p}} \Vert \n \tilde u\Vert_{L^p}^{p-1} \Vert \p^\beta(\tilde u-\tilde h)\Vert_{L^p}\\
& \leq C \Vert \omega_M(|\n \tilde u|)\Vert_{L^{r_p}} \left( \int_{Q} V_p(\n \tilde u) \dd x \right)^{\frac{p-1}{p}} \left(\int_{Q}V_p(\p^\beta (\tilde u-\tilde h)) \dd x \right)^{1/p}\\
& \leq \ep \int_{Q}V_p(\p^\beta (\tilde u-\tilde h) ) \dd x + C \Vert \omega_M(|\n \tilde u|)\Vert_{L^{r_p}}^{\frac{p}{p-1}}\int_{Q} V_p(\n \tilde u) \dd x.
\end{align*}
By the same reasoning where we similarly define $2_\beta$ as a higher integrability exponent, but this time using $\frac{1}{2'}+\frac{1}{2_\beta}+\frac{1}{r_2}=1$ we can obtain 
$$(III) \leq \ep \int_{Q}V_p(\p^\beta (\tilde u-\tilde h) ) \dd x + C \Vert \omega_M(|\n \tilde u|)\Vert_{L^{r_2}}^2\int_{Q} V_p(\n \tilde u) \dd x .$$
Absorbing the terms multiplied by $\ep$ to the left-hand side, and using that $\omega_M \leq 1$ we altogether obtain for some $r>1$: 
$$\int_{Q}V_p(\p^\beta (\tilde u-\tilde h)) \dd x\leq  C\left(\int_{Q}V_p(\n \tilde u) \dd x \right)^{p/q}+ C\Vert \omega_M(|\n \tilde u|)\Vert_{L^r}\int_{Q} V_p(\n \tilde u) \dd x.$$
It remains to observe, writing $\rho(t)=\omega_M \circ V_p^{-1}(t)$ which is concave with $ \lim_{t\searrow 0}\rho(t) =0$, that

$$\Vert \omega_M(|\n \tilde u|)\Vert_{L^r}^r \leq  C\strokedint_{Q}\omega_M(|\n \tilde u|) \dd x \leq C \rho\left( \strokedint_{Q} V_p(\n \tilde u) \dd x\right).$$
Thus by taking $H(t)=t^{p/q} + t\rho(t)^{1/r}$ if $p>2$ and $H(t)=t\rho(t)^{1/r}$ if $p=2$, we are finished.
\end{proof}
We are now ready to derive the estimates for the excess. Given $x_0 \in \Omega$, we define for concentric anisotropic boxes 
\begin{equation}
  E_R(\n u ) = \strokedint_{Q_R}V_p(\n u - (\n u)_R) \dd x \approx \strokedint_{Q_R}|\n u-(\n u)_R|^2+|\n u-(\n u)_R|^p \dd x.
\end{equation}

\begin{theorem}\label{thm excess decay estimates}
For every $M>0, \theta \in (0,1)$, there exist $\tau \in (0,1)$ and $\ep \in (0,1)$ such that if $Q_R \Subset \Omega$ and
$u \in \WW^{\mbf{a},p}_{\text{loc}}(\Omega ,\RN)$ is a local minimiser of $\mathcal{F}$ such that 
\begin{equation}\label{eqnsmallnessassumptions}E_R(\n u) <\ep \ \ \ \text{ and } \ \ \ |(\n u)_R|<M,
\end{equation}
then
\begin{equation}E_{\tau R}(\n u)\leq \tau^{(1+\theta)/\max a_i}E_R(\n u).
\end{equation}
\end{theorem}
\begin{proof}
Let $a$ be any polynomial such that $\n a=(\n u)_R$. Define $\tilde u= u-a$ and construct $\tilde h$ as in Proposition \ref{prop A harmonic approx}.
Set $h = \tilde h+a$, so that $\tilde u -\tilde h=u-h$. For $\tau \in (0,1/4)$ specified later, take $b$ to be the Poincaré polynomial for $h$
adapted to $Q_{2\tau R}$ satisfying $\n b= (\n h)_{2\tau R}.$ \\
We claim that $|\n b| \leq C_M$. Indeed, note that $\n h = \n \tilde h+(\n u)_R$, so it suffices to show $|(\n \tilde h)_{2\tau R}| \leq C_M$.
To see this, recall that $\tilde h \in \tilde u + W_0^{\mbf{a},p}(Q_R)$ satisfies $-\n^* F''(z_0) \n \tilde h=0$ in $Q_R.$
Thus, combining decay estimates (Lemma \ref{lemmaZeroRHSEstimate}) and $L^2$ estimates (a consequence of Theorem \ref{propLpBoundsLinear}) we have
\begin{align*}
|(\n \tilde h)_{2\tau R}| &\leq \strokedint_{Q_{2\tau R}} |\n \tilde h| \dd x \\
& \leq \left( \strokedint_{Q_{2\tau R}}|\n \tilde h|^2 \dd x\right)^{1/2}\\
&\leq C_M \left(\strokedint_{Q_R}|\n \tilde h|^2 \dd x\right)^{1/2}\\
& \leq C_M \left(\strokedint_{Q_R}|\n \tilde u|^2 \dd x \right)^{1/2}\\
& \leq C_M \left( E_R(\n u) \right)^{1/2} \\
& \leq C_M.
\end{align*}
Since $|\n b| \leq C_M$ we may use it in Caccioppoli's inequality (\ref{eqnCaccioppoli}) to estimate
\begin{align*}
E_{\tau R}(\n u) &=\strokedint_{Q_{\tau R}} V_p(\n u-(\n u)_{\tau R}) \dd x \\
&\lesssim \strokedint_{Q_{\tau R}} V_p(\n u- \n b) \dd x\\
& \lesssim_M \strokedint_{Q_{2\tau R}} \sum_{0 \neq \beta \leq \alpha, \alpha \in \mathcal{H}} V_p\left( \frac{\p^{\alpha-\beta}(u-b)}{(\tau R)^{\langle \beta, \mbf{a}^{-1} \rangle}}\right) \dd x\\
& \lesssim \strokedint_{Q_{2\tau R}} \sum_{0 \neq \beta \leq \alpha, \alpha \in \mathcal{H}} V_p\left( \frac{\p^{\alpha-\beta}(\tilde u-\tilde h)}{(\tau R)^{\langle \beta, \mbf{a}^{-1} \rangle}}\right) \dd x+\strokedint_{Q_{2\tau R}} \sum_{0 \neq \beta \leq \alpha, \alpha \in \mathcal{H}} V_p\left( \frac{\p^{\alpha-\beta}(h-b)}{(\tau R)^{\langle \beta, \mbf{a}^{-1} \rangle}}\right) \dd x\\
&:=(I)+(II).
\end{align*}
To estimate $(I)$ we apply Proposition \ref{prop A harmonic approx} to estimate
\begin{align*}
(I) &\leq C \tau^{-|\mbf{a}^{-1}|-p} \strokedint_{Q_R} \sum_{\langle \beta, \mbf{a}^{-1} \rangle<1}V_p\left(\frac{\p^\beta(\tilde u-\tilde h)}{ R^{1-\langle \beta, \mbf{a}^{-1}\rangle}} \right) \dd x\\
& \leq C \tau^{-|\mbf{a}^{-1}|-p} H \left(\strokedint_{Q_R} V_p(\n \tilde u) \dd x \right) = C \tau^{-|\mbf{a}^{-1}|-p} H(E_R(\n u)).
\end{align*}
To estimate $(II)$ we use the linear decay estimates (Lemma \ref{lemmaZeroRHSEstimate}) and the Poincaré inequality to estimate
\begin{align*}
  (II) & \lesssim \strokedint_{Q_{2\tau R}} \sum_{\langle \beta, \mbf{a}^{-1} \rangle <1} \left|\frac{\p^\beta (h-b)}{(\tau R)^{1-\langle \beta, \mbf{a}^{-1} \rangle}} \right|^2
  + \left|\frac{\p^\beta (h-b)}{(\tau R)^{1-\langle \beta, \mbf{a}^{-1} \rangle}} \right|^p \dd x\\
& \leq C \strokedint_{Q_{2\tau R}} |\n h-(\n h)_{2\tau R}|^2 +|\n h-(\n h)_{2\tau R}|^p \dd x\\
& \leq C\tau^{2/\max a_i}\strokedint_{Q_R} |\n h-(\n h)_{R}|^2+|\n h-(\n h)_{R}|^p \dd x\\
& \leq C\tau^{2/\max a_i}\strokedint_{Q_R} |\n h-(\n u)_R|^2 +|\n h-(\n u)_R|^p\dd x.
\end{align*}
Recall that $-\n^* F''(z_0) \n  h=0$ and $h \in u+ W_0^{\mbf{a},p}(Q_R)$. From Theorem \ref{propLpBoundsLinear} with a simple perturbation argument
(see \cite[Chapter 10]{GG}) this implies $\int_{Q_R} |\n h-\lambda|^q \dd x \leq C \int_{Q_R}|\n u-\lambda|^q \dd x$ for all $2 \leq q \leq p$.
Thus it follows that $(II) \leq C_M \tau^{2/\max a_i}E_R(\n u)$. Combining these estimates, we see that 
$$
E_{\tau R}(\n u) \leq C_M \tau^{-|\mbf{a}^{-1}|-p} H(E_R(\n u))+C_M \tau^{2/\max a_i}E_R(\n u).
$$
Now we choose $\tau \in (0,1/4)$ small enough that $C_M \tau^{2/\max a_i} \leq \frac{1}{2}\tau^{(1+\theta)/ \max a_i}$.
Then we choose $\ep \in (0,1)$ such that $C_M \tau^{-|\mbf{a}^{-1}|-p} H(t)/t \leq \frac{1}{2}\tau^{(1+\theta)/\max a_i}$ for all $t \leq \ep$.
We conclude that
$$
E_{\tau R}(\n u) \leq \tau^{(1+\theta)/\max a_i}E_R(\n u).
$$ 

\end{proof}
\begin{proof}[Proof of Theorem \ref{ThmPartialregularity}]
Let $0<\theta<1, M>0$. By a standard iteration procedure, there exists $\ep>0$ such that if \eqref{eqnsmallnessassumptions} holds, then
there exists $C$ such that $E_{r}(\n u) \leq C \left(\frac{r}{R} \right)^{\frac{1+\theta}{\max a_i}}E_R(\n u)$ for all $r<R$.
By continuity of the maps $x \rightarrow E_R(\n u;Q_R(x))$ and $x \rightarrow (\n u)_{Q_R(x)}$, this estimate holds for all boxes centred at $y$ in
some neighbourhood of $x$. Hence we deduce that $\n u \in \mathfrak{L}_{\mbf{a}}^{2, 1+\frac{1+\theta}{|\mbf{a}^{-1}| \max a_i}} \approx \CC^{0, \frac{1+\theta}{2}}_{\delta_{\mbf{a}}}$
on a neighbourhood of $x$. Finally, if $x$ is not in the singular set defined in \eqref{eqnSing1}, \eqref{eqnSing2} then there exists some $R>0$ such
that $|(\n u)_R|<M$ and $E_R(\n u) <\ep$ so that \eqref{eqnsmallnessassumptions} holds. By the Lebesgue differentiation theorem the singular set has measure zero, so we are done.
\end{proof}
\subsection{Extension to lower order derivatives} Theorem \ref{ThmPartialregularity} asserts a partial regularity result of the $\mbf{a}$-gradient,
but does not assert anything about the lower order derivatives. We briefly sketch how to obtain regularity of the lower order terms. Our method relies on
integral representations which can be found in \cite{Besovintrep}. We could alternatively deduce such regularity from \cite[Theorem 1]{boman}.
Given a rectangular domain $R$, one can construct a kernel $\Phi \in C_c^\infty(R)$ with
$\int_{\Rn} \! \Phi \, \dd x =1$, such that if we denote
$$
\Phi_\ep(x) = \frac{1}{\ep^{|\mbf{a}^{-1}|}}  \Phi\left( \frac{-x }{\ep^{\mbf{a}^{-1}}}\right) = \frac{1}{\ep^{|\mbf{a}^{-1}|}}\Phi\left(\frac{-x_1}{\ep^{a_1^{-1}}},\dots,\frac{-x_n}{\ep^{a_n^{-1}}}\right)
$$
and define the mollification
$$
f_\ep(x)=f*\Phi_{\ep}
$$
then we have the representation 
\begin{equation}\label{eqn integral representation formula}
  f_{\tau_1}(x)=f_{\tau_2}(x)-\sum_{i=1}^n\int_{\tau_1}^{\tau_2}\int_{\Rn} \! \p_i^{a_i}f(y)G_{i}^{\ep}(y-x) \, \dd y \, \dd \ep
\end{equation} 
for $\tau_1<\tau_2$, where $G_i \in \CC_c^{\infty}(R)$ and $G_i^\ep (x)=\frac{1}{\ep^{|\mbf{a}^{-1}|}}G_i\left(\frac{x}{\ep^{\mbf{a}^{-1}}}\right)$.
This is valid for any $f \in \WW^{\mbf{a},{p}}_\text{loc}(\R^n)$. Given $\langle \beta, \mbf{a}^{-1} \rangle <1$, we write
$$
H_{i,\tau_1,\tau_2}^\beta(x)=\int_{\tau_1}^{\tau_2}\ep^{-\langle \beta, \mbf{a}^{-1} \rangle-|\mbf{a}^{-1}|}|\p^\beta G_i \left(\frac{-x}{\ep^{\mbf{a}^{-1}}}\right)| \dd \ep.
$$
Then a simple calculation gives us the pointwise estimate
$$
|\p^\beta f_{\tau_1}(x) - \p^\beta f_{\tau_2}(x)| \leq \sum_{i=1}^n |\p_i^{a_i}f|*H_{i,\tau_1,\tau_2}^\beta(x).
$$
Now one can check that
$$
\Vert H_{i,\tau_1,\tau_2}^\beta \Vert_{\LL^1(\R^n)} \leq C \left(\tau_2^{1-\langle \beta, \mbf{a}^{-1}\rangle }-\tau_1^{1-\langle \beta, \mbf{a}^{-1} \rangle } \right)
$$
and hence we have 
$$
\Vert \p^\beta f_{\tau_1}-\p^\beta f_{\tau_2}\Vert_{\LL^\infty(\R^n)}\leq C  \Vert \n f\Vert_{\LL^\infty(\R^n)}\left(\tau_2^{1-\langle \beta, \mbf{a}^{-1}\rangle }-\tau_1^{1-\langle \beta, \mbf{a}^{-1} \rangle } \right).
$$
Assuming that $\n f \in \LL^\infty$, this estimate implies $(\p^\beta f_\ep)_{\ep>0}$ is Cauchy as $\ep \searrow 0$, so by uniform convergence, we see that
$\p^\beta f$ is a continuous function satisfying the estimate
\begin{equation}\label{eqn mollficaiton estimate boundedness}
\Vert \p^\beta f-\p^\beta f_\ep\Vert_{\LL^\infty(\Rn)} \leq C \ep^{1-\langle \beta, \mbf{a}^{-1} \rangle} \Vert \n f\Vert_{\LL^\infty(\Rn)}.
\end{equation}
We record another useful estimate in the case $\beta \neq 0$. Since $G_i$ is compactly supported, after differentiating the integral in
\eqref{eqn integral representation formula} and subtracting a constant, and this lets us deduce 
\begin{equation}\label{eqn mollification estimate holder}
|\p^\beta f_{\tau_1}(x) - \p^\beta f_{\tau_2}(x)| \leq \sum_{i=1}^n \int_{\R^n}|\p_i^{a_i}f(y)-\p_i^{a_i}f(x)|H_{i,\tau_1,\tau_2}^\beta(x-y) \dd y.\
\end{equation}
The key idea we utilise, as observed in \cite{LamyMironescu15}, is that one can characterise the smoothness of a function by the rate of convergence of mollification.
It is convenient to work with the spaces
$$
\CC_j^{k,\theta}(\R^n)=\left\{ u \in \LL^\infty(\Rn): \, \p^k_j u\in \LL^\infty(\Rn) \text{ and } \sup_{0<|h|\leq 1}\frac{\Vert \Delta_j^h \p_j^ku \Vert_{\LL^\infty}}{|h|^\theta} <\infty \right\}
$$
for $0<\theta \leq 1$ and $k \in \N_0$. Here, $\Delta_j^hf(x)=f(x+he_j)-f(x)$ denotes the usual differences. In other words, $\CC_j^{k,\theta}$ is the set of functions whose
$k^\text{th}$ derivative in direction $x_j$ is Hölder (resp. Lipschitz) continuous in direction $x_j$ with exponent $\theta$. Note that if $0<\theta<1$, then these spaces
are the higher order anisotropic Nikolskii spaces $N_j^{k+\theta,\infty}$, but it is more convenient to work with Lipschitz functions in the end point case $\theta=1$.

\begin{lemma}\label{lemma rate of convergence of mollifiers}
Suppose $f \in \LL^\infty(\Rn)$ satisfies
$$
\Vert f_\ep-f \Vert _{\LL^\infty}\leq C \ep^\lambda
$$
for all $0<\ep<\ep_0$, where $0<\lambda<\frac{1}{a_j}$. Then $f \in \CC_j^{0,\lambda a_j}(\Rn)$.
\end{lemma}
\begin{proof}
We first claim that $\Vert \p_j f_\ep\Vert_{\LL^\infty} \leq C\ep^{\lambda-\frac{1}{a_j}}$ (where $C$ may depend on $\ep_0$ and $f$).
Set $L=\Vert \Phi \Vert_{\LL^1}$ and choose some number $b>1$ so that $Lb^{\lambda-\frac{1}{a_j}}<1$.
Then there exists $N$ such that $b^N \ep <\ep_0 \leq b^{N+1}\ep.$ By writing
$$
f_\ep = \Phi_\ep *(f-f_{b\ep})+\Phi_\ep*f_{b\ep}
$$
we can estimate
\begin{align*}
\Vert \p_j f_\ep\Vert_{\LL^\infty} &\leq \Vert \p_j\Phi_\ep \Vert _{\LL^1} \Vert f-f_{b\ep} \Vert_{\LL^\infty}+\Vert \Phi_\ep \Vert_{\LL^1}\Vert \p_jf_{b\ep}\Vert_{\LL^\infty} \\
& \leq C\ep^{-1/a_j}(b \ep)^\lambda+L\Vert \p_j f_{b\ep}\Vert_{\LL^\infty}
\end{align*}
Iterating this we see that 
\begin{align*}
\Vert \p_j f_\ep\Vert_{\LL^\infty} \leq C \ep^{\lambda-1/a_j} \sum_{k=1}^N(Lb^{\lambda-\frac{1}{a_j}})^{k} + L^N\Vert \p_j f_{b^N\ep}\Vert_{\LL^\infty}
\end{align*}
from which we see that $\Vert \p_j f_\ep\Vert_{\LL^\infty} \leq C\ep^{\lambda-\frac{1}{a_j}}$ as claimed. Now using this estimate, we have
\begin{align*}
\Vert \Delta_j^hf\Vert_{\LL^\infty} &\leq  \Vert \Delta_j^h(f-f_\ep)\Vert_{\LL^\infty}+\Vert \Delta_j^hf_\ep\Vert_{\LL^\infty}\\
& \leq 2\Vert f-f_\ep \Vert_{\LL^\infty}+|h|\Vert \p_jf_\ep\Vert_{\LL^\infty}\\
& \leq C\ep^{\lambda}+|h|\ep^{\lambda-1/a_j}.
\end{align*}
Balancing these terms we set $\ep=|h|^{a_j}$ to conclude
$$
\Vert \Delta_j^h f \Vert_{\LL^\infty} \leq C|h|^{\lambda a_j}
$$
for all $|h|$ sufficiently small.
\end{proof}
Combining this lemma with our earlier remarks, we have the following. For convenience we adopt the notation $\CC^{k,0}_j = \CC^{k-1,1}_j$ for $k \geq 1$.
\begin{prop}
Let $u \in \WW^{\mbf{a},p}(\Rn)$.
\begin{enumerate}[(i)]
\item 
Suppose that $\n u \in \LL^\infty(\Rn)$. For each $\langle \beta, \mbf{a}^{-1} \rangle <1$, if $s_{\beta,j}=a_j(1-\langle \beta, \mbf{a}^{-1} \rangle)$, then 
$$
\p^\beta u \in \bigcap_{j=1}^n C_j^{\lfloor s_{\beta,j}\rfloor, \left\{s_{\beta,j}\right\}}.
$$
\item Suppose that for each direction $j$, $\n u$ is Hölder continuous in direction $x_j$ with exponent $a_j \lambda$
(that is, $\n u \in \CC_{\delta_\mbf{a}}^{\lambda \max a_i}$ for some $0<\lambda<\frac{1}{\max a_i}$). For each $\langle \beta, \mbf{a}^{-1} \rangle < 1$, define the 
$\tilde{s}_{\beta,j} = a_j(1 - \langle \beta, \mbf{a}^{-1} \rangle+ \lambda)$. Assume that $\left\{\tilde s_{\beta,j} \right\} \neq 0$ for each $j$. 
Then
$$
\p^\beta u \in \bigcap_{j=1}^n C_j^{\lfloor \tilde{s}_{\beta,j}\rfloor, \left\{\tilde{s}_{\beta,j}\right\}}.
$$
\end{enumerate}
\end{prop}
\begin{proof}
$(i)$ follows immediately from Lemma \ref{lemma rate of convergence of mollifiers} and \eqref{eqn mollficaiton estimate boundedness},
noting that in the case where the fractional part is zero, we may integrate directly to deduce Lipschitz regularity. To prove $(ii)$,
we consider the case $\beta \neq 0$ as otherwise, the result follows from integrating the first order derivatives. At scale $\ep$, we have
$$
|\p_i^{a_i}u(y)-\p_i^{a_i}u(x)| \leq C \sum_{k=1}^n |x_k-y_k|^{a_k\lambda} \leq C \ep^{\lambda }.
$$
Therefore by taking the support of $\Phi$ to be the unit cube in the construction of the integral representations, it follows from
\eqref{eqn mollification estimate holder} that
$$
\Vert \p^\beta u-\p^\beta u_\ep\Vert_{\LL^\infty} \leq C \ep^{1-\langle \beta, \mbf{a}^{-1} \rangle +\lambda}.
$$
Once more applying Lemma \ref{lemma rate of convergence of mollifiers} gives the result.
\end{proof}
\end{section}

\end{document}